\documentclass[12pt]{article}
 
 \usepackage{amsmath,amssymb,amscd,amsthm,esint}
 
\usepackage{graphics,amsmath,amssymb,amsthm,mathrsfs}

\newtheorem{theorem}{Theorem}[section]
\newtheorem{lemma}[theorem]{Lemma}

\theoremstyle{definition}

\theoremstyle{remark}
\newtheorem{remark}[theorem]{Remark}

\numberwithin{equation}{section}

\newcommand{\norm}[1]{\lVert#1\rVert}

\newcommand{\Dr}[2]
{\bigg(\begin{matrix} #1 \\ #2  \end{matrix}\bigg)}

\newcommand{\cL}{\mathcal{L}}

\newcommand{\R}{\mathbb{R}}

\newcommand{\bS}{\mathbb{S}}
\newcommand{\bH}{\mathbb{H}}
\newcommand{\Z}{\mathbb{Z}}

\newcommand{\T}{\mathbb{T}}

\newcommand{\e}{\varepsilon}

\newcommand{\txt}[1]{\text{#1}}

\begin{document}

\title{Regularity of Homogenized Boundary Data\\ 
in Periodic Homogenization of Elliptic Systems}


\author{
Zhongwei Shen\thanks{Supported in part by NSF grant DMS-1600520.}
\qquad Jinping Zhuge \thanks{Supported in part by NSF grant DMS-1600520.}}

%    author one information
%\author{Zhongwei Shen, Jinping Zhuge}
%\address{Department of Math, University of Kentucky, Lexington, KY, 40506, USA.}
%\curraddr{}
%\email{jinping.zhuge@uky.edu}
%\thanks{The author is supported in part by National Science Foundation grant DMS-1161154.}
%\date{\today}

%\subjclass[2010]{35B27.}
%    The 2010 edition of the Mathematics Subject Classification is
%    now available.  If you are citing a classification from the
%    new scheme, use the following input coding instead.
%\subjclass[2010]{Primary }

\date{ }

\maketitle

%\pagestyle{plain}
%\tableofcontents
%    Input .tex files
%\input{s0.tex}

\begin{abstract}
This paper is concerned with periodic homogenization of second-order elliptic systems in divergence form 
with oscillating Dirichlet data or Neumann data of first order. 
We prove that  the homogenized boundary data  belong to $W^{1, p}$ for any
$1<p<\infty$.
In particular, this implies that the boundary layer tails are H\"older continuous of order
$\alpha$ for any $\alpha \in (0,1)$.

\medskip

\noindent{\it MSC2010:} 35B27, 74Q05.

\medskip

\noindent{\it Keywords:} Homogenization; Boundary Layers; Oscillating Boundary Data.

\end{abstract}

%\keywords{Homogenization}

\section{Introduction}

In this paper we consider the uniformly elliptic operator with periodically oscillating coefficients
\begin{equation*}
	\cL_\e = -\txt{div} (A(x/\e) \nabla) = - \frac{\partial}{\partial x_i} \bigg\{ a^{\alpha\beta}_{ij} \Big( \frac{x}{\e}\Big) \frac{\partial}{\partial x_j}\bigg\},
\end{equation*}
where $1\le i,j\le d$, $1\le \alpha,\beta\le m$, and $0<\e \le 1$. 
Throughout we assume that the coefficient matrix $A$ satisfies the following conditions:
\begin{itemize}

\item

Ellipticity: there exists $\lambda>0$ such that
\begin{equation}\label{ellipticity}
	\lambda |\xi|^2 \le a_{ij}^{\alpha\beta} \xi_i^\alpha \xi_j^\beta \le \lambda^{-1}|\xi|^2 \quad \txt{for any } \xi = (\xi_i^\alpha) \in \R^{m\times d};
\end{equation}

\item

Periodicity: $A$ is 1-periodic, that is
\begin{equation}\label{periodicity}
	A(y+z) = A(y) \quad \txt{for any }  y\in \R^d \text{ and } z\in \Z^d;
\end{equation}

\item

Smoothness: 
\begin{equation}\label{smoothness}
a_{ij}^{\alpha\beta} \in C^\infty(\T^d) \quad \text{ for } 1\le \alpha, \beta\le m \text{ and }
1\le i,j\le d.
\end{equation}

\end{itemize}
Recently, there has been considerable interest in the study of homogenization of Dirichlet problem
with oscillating boundary data,
\begin{equation}\label{DP-0}
\left\{
\aligned
\mathcal{L}_\e (u_\e) & =0 & \quad &\text{ in } \Omega,\\
u_\e (x)  & = f(x, x/\e) &\quad &\text{ on } \partial\Omega,
\endaligned
\right.
\end{equation}
where $f(x, y)$ is 1-periodic in $y$ 
\cite{GerMas11, GerMas12, Pra13, KLS14,
ASS-2013, ASS-2014, Ass-2015, Alek16, AKMP16, ShenZhuge16, Z17} 
(also see earlier work in \cite{MV97, Moskow-Vogelius-2, AA99}
as well as related work for nonlinear elliptic equations in \cite{Feldman-2014, Choi-Kim-2014, Kim-2015, FK17}).
In particular, under the assumption 
that $\Omega$ is a smooth and strictly convex domain in $\R^d$,
 it was proved in \cite{GerMas12} that the homogenized problem for (\ref{DP-0}) is given 
by 
\begin{equation}\label{DP-h}
\left\{
\aligned
\mathcal{L}_0 (u_0) & =0 & \quad &\text{ in } \Omega,\\
u_0  & = \overline{f} &\quad &\text{ on } \partial\Omega,
\endaligned
\right.
\end{equation}
where $\mathcal{L}_0$ is the usual homogenized operator and
$\overline{f} $ is a function whose value at $x\in \partial\Omega$ depends only on 
$A$, $f(x, \cdot)$ and the outward normal $n$ to $\partial\Omega$ at $x$.
Moreover, a convergence rate for $\| u_\e -u_0\|_{L^2(\Omega)}$
was established in \cite{GerMas12}.
The near sharp convergence rates were obtained in \cite{AKMP16} for $d\ge 4$ 
and in \cite{ShenZhuge16} for $d=2$ or $3$.
Furthermore, the present authors in \cite{ShenZhuge16} considered the Neumann problem
with first-order oscillating boundary data,
\begin{equation}\label{NP-0}
\left\{
\aligned
\mathcal{L}_\e (u_\e) & =0  &\quad& \text{ in } \Omega,\\
\frac{\partial u_\e}{\partial \nu_\e}  &=
T_{ij} \cdot \nabla_x \big\{ g_{ij}(x, x/\e) \big\} 
&\quad & \text{ on } \partial\Omega,
\endaligned
\right.
\end{equation}
where $T_{ij}=n_i e_j -n_j e_j$ is a tangential vector field on $\partial\Omega$
and $\{ g_{ij} (x, y)\} $ are 1-periodic in $y$.
It was proved in \cite{ShenZhuge16} that 
if $\Omega$ is smooth and strictly convex,
the homogenized problem for (\ref{NP-0}) is given 
by 
\begin{equation}\label{NP-h}
\left\{
\aligned
\mathcal{L}_0 (u_0) & =0  &\quad& \text{ in } \Omega,\\
\frac{\partial u_0}{\partial \nu_0}  &=
T_{ij} \cdot \nabla_x  \overline{g}_{ij}
&\quad & \text{ on } \partial\Omega,
\endaligned
\right.
\end{equation}
where $\frac{\partial u_0}{\partial\nu_0}$ denotes the conormal derivative of $u_0$ associated with
 $\mathcal{L}_0$, and $\{ \overline{g}_{ij} \}$ 
are functions on $\partial\Omega$ whose value at $x\in \partial\Omega$
depend only on $A$, $\{ g_{ij}(x, \cdot )\}$ and $n(x)$.
Assume that $\int_\Omega u_\e =\int_\Omega u_0=0$.
The near optimal rate of convergence 
for $\| u_\e -u_0\|_{L^2(\Omega)}$ was also established in \cite{ShenZhuge16}
for $d\ge 3$.
In \cite{Z17} the second author investigated the case of non-convex domains and 
extended the results in \cite{AKMP16, ShenZhuge16} for Dirichlet problems
to certain domains of finite type.
We point out that one of main motivations for studying 
boundary value problems (\ref{DP-0}) and (\ref{NP-0}) with oscillating data
is its applications to the higher-order convergence in the two-scale expansions
of solutions to boundary value problems with non-oscillating boundary data.

Our primary concern in this paper is the regularity of the homogenized data
$\overline{f}$ in (\ref{DP-h}) and
$\{ \overline{g}_{ij} \}$ in (\ref{NP-h}).
Under the assumption that $\Omega$ is smooth and strictly convex,
it was proved in \cite{GerMas12} that 
$\nabla_{\tan} \overline{f}\in L^{p, \infty}(\partial\Omega)$ with $p=\frac{d-1}{2}$.
The result was improved in \cite{AKMP16} to
$\nabla_{\tan}  \overline{f} \in L^{p, \infty}(\partial\Omega)$
with $p=\frac{2(d-1)}{3}$ if $d\ge 3$, and to $\overline{f}
 \in W^{1, p}(\partial\Omega)$ for any $p<\frac{2}{3}$ if $d=2$.
 Further improvement was made in \cite{ShenZhuge16}, where we proved that
 $\overline{f}\in W^{1, p}(\partial\Omega)$ for any $p<d-1$ and $d\ge 2$.
 In \cite{ShenZhuge16} we also obtained the regularity estimate
 $\overline{g}_{ij}\in W^{1, p}(\partial\Omega)$ for the Neumann problem 
 (\ref{NP-0}), where $p<d-1$ and $d\ge 3$.

The following two theorems are the main results of this paper.

\begin{theorem}[Dirichlet Data]\label{main-thm-01}
Assume that $A$ satisfies (\ref{ellipticity}), (\ref{periodicity}) and (\ref{smoothness}).
Let $\Omega$ be a smooth and strictly convex domain in $\R^d$.
Let $\overline{f}$ denote the homogenized data in (\ref{DP-h}).
Then
\begin{equation}\label{main-estimate-01}
\| \overline{f}\|_{W^{1, p}(\partial\Omega)}
\le C_p \left(\int_{\T^d}  \| f(\cdot, y)\|^2_{C^1(\partial\Omega)}\, dy\right)^{1/2}
\quad \text{ for any } 1<p<\infty,
\end{equation}
where $C_p$ depends only on $d$, $m$, $\lambda$, $p$, and
$\| A\|_{C^k(\T^d)}$ for some $k=k(d, p)>1$.
\end{theorem}

\begin{theorem}[Neumann Data]\label{main-thm-02}
Assume that $A$ satisfies (\ref{ellipticity}), (\ref{periodicity}) and (\ref{smoothness}).
Let $\Omega$ be a smooth and strictly convex domain in $\R^d$.
Let $\overline{g}=( \overline{g}_{ij}) $ denote the homogenized data in (\ref{NP-h}).
Then
\begin{equation}\label{main-estimate-02}
\| \overline{g}\|_{W^{1, p}(\partial\Omega)}
\le C_p \left(\int_{\T^d} \| g(\cdot, y)\|^2_{C^1(\partial\Omega)}\, dy\right)^{1/2}
\quad \text{ for any } 1<p<\infty,
\end{equation}
where $C_p$ depends only on $d$, $m$, $\lambda$, $p$, and
$\| A\|_{C^k(\T^d)}$ for some $k=k(d, p)>1$.
\end{theorem}

It follows from regularity estimates (\ref{main-estimate-01}) and (\ref{main-estimate-02}) that
the homogenized data $\overline{f}$ and $\overline{g}=(\overline{g}_{ij})$ are H\"older continuous
of order $\alpha$ for any $\alpha\in (0,1)$.
We should point out that the assumption that $\Omega$ is strictly convex is
not essential for Theorems \ref{main-thm-01} and \ref{main-thm-02}.
In fact, the proof goes through as long as one has
$[\varkappa (n(x))]^{-1} \in L^q (\partial\Omega)$ for some $q>0$ (see (\ref{cdn_Dioph})
for the definition of $\varkappa$).
Consequently, the conclusions of Theorems \ref{main-thm-01}
and \ref{main-thm-02} continue to hold for the domains of finite type considered in \cite{Z17}.

We mention several related work regarding the continuity of homogenized data.
In \cite{Alek-16}, under the additional assumption that $A$ is independent of some
rational direction $\nu_0$, it was proved that the homogenized Dirichlet data has a 
unique continuous extension to the set $\{ x\in \partial\Omega: n(x)\cdot \nu_0
\neq 0\}$.
The problem of H\"older continuity
was also studied in \cite{Choi-Kim-2014, FK17} for second-order nonlinear elliptic equations of form
$F(D^2 u_\varepsilon, x/\varepsilon)=0$.
In particular, it was shown in \cite{FK17} that if the homogenized operator
$\overline{F}$ is either rotational invariant or linear,
then the homogenized Dirichlet data is H\"older continuous, and that
 the homogenized data may be discontinuous in general.
 Note  that the linear elliptic equations in non-divergence form may be written in a divergence form
with $\text{div}(A)=0$.
In this case, the first-order correctors are trivial and as a result, it is easy to see that
the homogenized data is smooth if $\Omega$ is smooth and 
satisfies some geometric conditions.
As far as we know,  the continuity of the homogenized data
in the general case of elliptic equations in divergence form is not known previously.

We now describe our general approach to Theorems \ref{main-thm-01}
and \ref{main-thm-02} as well as some of the key estimates in the proof.
Our starting point for the proof of Theorem \ref{main-thm-01}
is a formula for the homogenized data $\overline{f}$ discovered  in \cite{AKMP16}.
See Theorem \ref{thm_mu}.
This formula reduces the problem to the study of the dependance on $n\in \mathbb{S}^{d-1}$ of
solutions $V_n=V_n(\theta, t)$ to the Dirichlet problem,
\begin{equation}\label{DP-01}
	\left\{
	\begin{aligned}
		-\Dr{N_n^T \nabla_\theta}{\partial_t}
		\cdot B_n \Dr{N_n^T \nabla_\theta}{\partial_t} V_n
		&= 0 &\quad & \txt{ in } \T^d\times \R_+, \\
		V_n (\theta, 0)&= \phi (\theta) &\quad & \txt{ on } \T^d\times \{0\},
	\end{aligned}
	\right.
\end{equation}
where $\phi\in C^\infty(\T^d; \R^m)$,
$B_n=B_n (\theta, t) =M_n^T A^* (\theta -tn) M_n$, and
$M_n$ is a $d\times d$ orthogonal matrix whose first $d-1$ columns
are given by $N_n$ and whose last column is $-n$.

To describe our key estimates, we need to introduce some notations.
A unit vector $n = (n_1, n_2, \cdots, n_d)$ 
$\in \bS^{d-1}$ is called rational if $n\in \R\Z^d$ and called irrational otherwise. 
Moreover, a unit vector $n$ is called Diophantine if there exists some constant $C > 0$ such that
\begin{equation}\label{cdn_Dioph}
	|(I - n\otimes n)\xi| \ge C |\xi|^{-2} \qquad \txt{for all } \xi\in \Z^d\setminus \{0\}.
\end{equation}
Denote by $\varkappa = \varkappa(n)$ the Diophantine constant, 
which is defined as the largest constant validating (\ref{cdn_Dioph}). 
We use $\bS^{d-1}_R$, $\bS^{d-1}_I $ and $ \bS^{d-1}_D$ 
to represent the sets of rational, irrational and Diophantine unit vectors, respectively. 
Note  that $\bS^{d-1}_D$ is a subset of $\bS^{d-1}_I$ and has full surface measure of $\bS^{d-1}$.

Let $n, \widetilde{n}\in \bS^{d-1}_D$. 
We will show in Section 2 that for any $\sigma\in (0,1)$,
\begin{equation}\label{main-001}
\left(\int_{\T^d} \big|\partial_t V_{n} (\theta, 0)
-\partial_t V_{\widetilde{n}} (\theta, 0)\big|^2\, d\theta\right)^{1/2}
\le C_\sigma \varkappa^{-\sigma} |n -\widetilde{n}|,
\end{equation}
where $\varkappa =\max\big\{ \varkappa (n), \varkappa(\widetilde{n})\big\}$
and $C_\sigma$ depends only on 
$d$, $m$, $\sigma$, $\lambda$, $\|A\|_{C^k(\T^d)}$ and $\|\phi\|_{C^k(\T^d)}$
for some $k=k(d, \sigma)>1$.
Theorem \ref{main-thm-01} follows from (\ref{main-001}) 
by using the representation formula mentioned above
and an approximation argument.

To prove (\ref{main-001}),
besides the energy estimates established in \cite{GerMas11, GerMas12, AKMP16},
one needs to fully take advantage of the fact that if
\begin{equation}\label{relation-u-V}
u^s (x)=V_n(x-(x\cdot n) n -sn, -x\cdot n -s),
\end{equation}
then $u^s$ is a solution of the Dirichlet problem in a half-space,
\begin{equation}\label{eq_Hu}
	\left\{
	\begin{aligned}
		\cL^*_1 (u^s) &= 0 &\quad & \txt{in } \bH^d_n (s), \\
		u^s &= \phi &\quad & \txt{on } \partial\bH^d_n(s),
	\end{aligned}
	\right.
\end{equation}
where $\bH^d_n(s)=\bH^d_n -sn$ and
$
\bH^d_n = \{x \in \R^d: x\cdot n < 0\} 
$
is the half-space whose boundary contains the origin and with outward normal $n$.
This allows us to apply the large-scale boundary regularity estimates for the
operator $\mathcal{L}^*_1$.
The technique was already used in \cite{GerMas12, AKMP16} to establish the boundedness of
$V_n$ and in \cite{ShenZhuge16} for a crucial weighted norm inequality.
Here, among other things, we apply the technique to establish the boundedness of
$\nabla_\theta V_n$ as well as some pointwise decay estimates for
$\partial_t V_n$ and $N_n^T\nabla_\theta V_n$.

We remark that the asymptotic behavior of the solution 
$u^s$ of (\ref{eq_Hu}) as $x\cdot n\to -\infty$ is well understood thanks to 
\cite{MV97, AA99, GerMas11, GerMas12, Pra13, Alek16}. 
In particular, 
if $n$ is irrational, it was shown in \cite{Pra13} that
 there exists a constant vector $\mu^*(n,\phi)\in \R^m$ independent of $s$ such that
\begin{equation}\label{eq_muus}
	\mu^*(n,\phi) = \lim_{x\cdot n \to -\infty} u^s(x),
\end{equation}
though the rate of convergence could be arbitrarily slow in general. 
On the other hand, if $n$ is rational \cite{MV97,AA99},
 the above limit depends  on $s$ and possesses  an exponential rate of convergence. 
The mapping $\mu: \bS^{d-1}_I\times C^\infty(\T^d;\R^m) \mapsto \R^m$ defined via (\ref{eq_muus}),
but with $\mathcal{L}_1^*$ replaced by $\mathcal{L}_1$,
is called the boundary layer tail (BLT) for Dirichlet problems associated with $\cL_1$.
It follows from  \cite{GerMas12} that
\begin{equation}\label{BLT}
\overline{f} (x) =\mu (n(x), f(x, \cdot)), \qquad \text{ if } n(x)\in \bS_D^{d-1}.
\end{equation}
Thus, by Theorem \ref{main-thm-01},
  $\norm{\mu (\cdot, \phi)}_{W^{1,p}(\bS^{d-1})} \le
 C \norm{\phi}_{L^2(\T^d)}$ for any $1<p<\infty$.
  Consequently, for any $0<\alpha<1$, $\mu(\cdot, \phi)$ extends to a H\"older
 continuous function of order $\alpha$ on $\bS^{d-1}$ and 
	\begin{equation}\label{est_muCalpha}
		|\mu(n, \phi) - \mu(\widetilde{n}, \phi)| \le  
		C_\alpha |n-\widetilde{n}|^\alpha \norm{\phi}_{L^2(\T^d)} 
		\quad \txt{for any } n,\widetilde{n}\in \bS^{d-1},
	\end{equation}
	where $C_\alpha $ depends only on $d$, $m$, $\alpha$ and $A$.
	
	Our approach to Theorem \ref{main-thm-02} for Neumann problems
	is similar to that used for Theorem \ref{main-thm-01}.
	The starting point is a formula for the homogenized data $\{ \overline{g}_{ij}\}$ obtained in \cite{ShenZhuge16}.
	See Theorem \ref{thm-N-formula}.
	As in the case of Dirichlet problems,
	this formula reduces the problem to the study of the dependence in $n\in \bS^{d-1}$ of solutions 
	$U_n=U_n (\theta, t)$ to the 
	Neumann problem,
	\begin{equation}\label{NP-01}
	\left\{
	\begin{aligned}
		-\Dr{N_n^T \nabla_\theta}{\partial_t}
		\cdot B_n \Dr{N_n^T \nabla_\theta}{\partial_t} U_n
		&= 0 &\quad & \txt{ in } \T^d\times \R_+, \\
		-e_{d+1}\cdot B_n \Dr{N_n^T \nabla_\theta}{\partial_t} U_n &= T_n \cdot \nabla_\theta \phi
		&\quad & \txt{ on }  \T^d\times \{0\},
	\end{aligned}
	\right.
\end{equation}
where $T_n\in \R^d$, $|T_n|\le 1$ and $T_n\cdot n=0$.
Let $n, \widetilde{n}\in \bS^{d-1}_D$. 
We will show in Section 3 that for any $\sigma\in (0,1)$,
\begin{equation}\label{main-002}
\left(\int_{\T^d} \big|\nabla_\theta U_{n} (\theta, 0)
-\nabla_\theta U_{\widetilde{n}} (\theta, 0)\big|^2\, d\theta\right)^{1/2}
\le C_\sigma \varkappa^{-\sigma} |n -\widetilde{n}|,
\end{equation}
where $\varkappa =\max\big\{ \varkappa (n), \varkappa(\widetilde{n})\big\}$
and $C_\sigma$ depends only on 
$d$, $m$, $\sigma$, $\lambda$, $\|A\|_{C^k(\T^d)}$ and $\|\phi\|_{C^k(\T^d)}$
for some $k=k(d, \sigma)>1$.
Theorem \ref{main-thm-02} follows from (\ref{main-002}) 
and the representation formula mentioned above.
Finally, we point out that the key estimates in the proof of (\ref{main-002})
rely on the observation that if
$u^s (x)=U_n (x-(x\cdot n) n -sn, -x\cdot n -s)$, 
then $u^s$ is a solution to the Neumann problem,
\begin{equation}
\mathcal{L}^*_1 (u^s)=0 \quad \text{ in } \mathbb{H}^d_n (s)
\quad \text{ and } \quad 
\frac{\partial u^s}{\partial \nu^*_1} = T_n\cdot \nabla_x \phi \quad \text{ on } \partial\mathbb{H}^d_n (s).
\end{equation}
We refer the reader to Section 3 for details.

%%%%%%%%%%%%%%%%%%%%%%%%%%%%%%%%%%%%%%%%%%%

\section{Regularity for Dirichlet problems}

Assume that $A$ satisfies conditions (\ref{ellipticity})-(\ref{smoothness}). 
For $1\le j\le d$ and $1\le \beta\le m$, 
let $\chi = (\chi_j^\beta) = (\chi_j^{1\beta},\chi_j^{2\beta},\cdots,\chi_j^{m\beta})$ 
denote the correctors for $\cL_\varepsilon$.
By definition they are 1-periodic functions satisfying the equation
 $\cL_1 (\chi^{\beta}_{j} + P_j^\beta) = 0$ in $\R^d$ with 
 $\int_{\T^d} \chi_j^\beta = 0$,
 where $P_j^\beta(x) = x_je^\beta $. 
 The homogenized operator is given by $\cL_0 = -\txt{div}(\widehat{A}\nabla)$,
  where the homogenized matrix $\widehat{A} = (\widehat{a}_{ij}^{\alpha\beta})$ is defined by
  
\begin{equation*}
	\widehat{a}_{ij}^{\alpha\beta} =  \int_{\T^d} \Big\{ a_{ij}^{\alpha\beta} + a_{ik}^{\alpha\gamma} \frac{\partial}{\partial y_k} (\chi_j^{\gamma\beta}) \Big\}.
\end{equation*}
We also introduce the adjoint operator $\cL^*_\varepsilon = - \txt{div}(A^*(x/\varepsilon)\nabla )$, 
where $A^* = (a_{ij}^{*\alpha\beta})$ with
 $a_{ij}^{*\alpha\beta} = a_{ji}^{\beta\alpha}$. 
 Note that $A^*$ also satisfies (\ref{ellipticity})-(\ref{smoothness}). 
 Let $\chi^*=(\chi^{*\beta}_j)$ denote the correctors for $\cL^*_\varepsilon$.

The solvability of the Dirichlet problem (\ref{eq_Hu}) is not obvious,
 since $\bH^d_n(s)$ is unbounded. 
Nevertheless, by using Lipschitz estimates in   \cite{AveLin87}
and an approximation argument,
one may establish the existence of the Poisson kernel 
in a half-space and hence the solvability of (\ref{eq_Hu}) via the Poisson integral formula. 

\begin{theorem}\label{thm_Poi}
	Let $\Omega=\bH^d_n(s)$ for some $n\in \bS^{d-1}$ and $s\in \R$.
	Then, for any bounded continuous function $\phi $ in $\R^d$, there exists a unique 
	bounded function $u$ in
	$C^\infty(\Omega; \R^m)\cap C(\overline{\Omega}; \R^m)$ such that
	\begin{equation}\label{DP}
	\mathcal{L}^*_1 (u)=0 \quad \text{ in } \Omega \quad \text{ and } \quad u=\phi \quad
	\text{ on } \partial\Omega.
	\end{equation}
	Moreover, the solution may be represented by
	\begin{equation}\label{eq_Pint}
		u(x) = \int_{\partial\Omega} P^*(x,y) \phi (y)\, d\sigma(y),
	\end{equation}
	where the Poisson kernel $P^* = P^*(x,y)$ satisfies
	\begin{equation}\label{est_Poison}
		|P^*(x,y)| \le \frac{C\, \delta(x)}{|x-y|^d},
			\end{equation}
	\begin{equation}\label{est_DPoison}
		|\nabla_x P^*(x,y)| \le \frac{C \min \big\{|x-y|, \delta(x) \big\}}{|x-y|^{d+1}}
	\end{equation}
	for any $x\in \Omega$ and $y\in \partial\Omega$,
 $\delta(x) = \emph{dist}(x, \partial\bH^d_n(s)) = |s+x\cdot n|$,
 and $C$ depends only on $d$, $m$, $\lambda$,  and
 some H\"older norm of $A$ on $\T^d$.	
\end{theorem}

\begin{proof}
	The theorem was proved in \cite[Proposition 2.5]{GerMas12}. 
	\end{proof}
	
\begin{remark}\label{remark-20}
	By the boundary Lipschitz estimates \cite{AveLin87} and the Cacciopoli inequality,
	the uniqueness holds under the sublinear growth condition:
	$|u(x)|\le C_0 (1+ \delta(x))^\alpha$ for some $C_0>0$ and $\alpha \in (0,1)$.
	Also, it follows readily from (\ref{est_Poison}) that the Miranda-Agmon maximum principle, 
	\begin{equation}\label{max-p}
	\| u\|_{L^\infty(\Omega)} \le C \| \phi\|_{L^\infty(\partial\Omega)}
	\end{equation}
	holds, where $C$ depends only on $d$, $m$, $\lambda$,  and
 some H\"older norm of $A$ on $\T^d$.	
\end{remark}

An alternative way to establish  the solvability of (\ref{eq_Hu})
for periodic data $\phi$ is to lift the problem to a ($d+1$)-dimensional problem in the upper half-space. 
Fix $n\in \bS^{d-1}$. Let $M = (N,-n)$ be a $d\times d$ orthogonal matrix such that the last column is $-n$ 
and the first $d-1$ column  is a $d\times (d-1)$ matrix $N$. 
Now we seek a solution $u$ of (\ref{eq_Hu}) in a particular form
\begin{equation}\label{eq_uV0}
	u^s(x) = V(x-(x\cdot n)n-sn , -x\cdot n-s).
\end{equation}
It is not hard to see that $V=V(\theta, t)$ has to satisfy the following lifted degenerate system,

\begin{equation}\label{eq_Vs}
	\left\{
	\begin{aligned}
		-\Dr{N^T \nabla_\theta}{\partial_t}
		\cdot B \Dr{N^T \nabla_\theta}{\partial_t} V
		&= 0 &\quad & \txt{ in } \T^d\times (0,\infty), \\
		V(\theta,0) &= \phi (\theta) &\quad & \txt{ on }  \T^d\times \{0\},
	\end{aligned}
	\right.
\end{equation}
where $B(\theta,t) = M^T A^*(\theta-tn) M$. Note that $M M^T = I$ implies $I = NN^T + n\otimes n$. 
It follows that
\begin{equation}
M \Dr{N^T \nabla_\theta}{\partial_t} = (I - n\otimes n) \nabla_\theta - n \partial_t.
\end{equation}
Thus, the solution $V$ is independent of the choice of $N$.

The well-posedness of (\ref{eq_Vs}) was given by \cite[Propositions 2.1 and 2.6]{GerMas12}.

\begin{lemma}\label{lemma-DKV}
 Let $n\in \bS^{d-1}$.
 Then,  for any $\phi\in C^\infty(\T^d;\R^m)$, 
 the system (\ref{eq_Vs}) has a  smooth solution $V =V(\theta, t)$ satisfying
	\begin{equation}\label{est_DDrV}
		\left(\int_0^\infty \int_{\T^d}\left( |N^T\nabla_\theta \partial_\theta^\alpha \partial_t^j V|^2 + |\partial_\theta^\alpha \partial_t^{1+j} V|^2 d\theta\right)\,  dt\right)^{1/2}
		 \le C \norm{\phi}_{C^{|\alpha|+j+1}(\T^d)},
	\end{equation}
	where $|\alpha|$, $j \ge 0$, and  $C$ 
	depends only on $d$, $m$, $|\alpha|$, $j$ and $A$. 
	Moreover, if $n\in \bS^{d-1}_D$ with Diophantine constant $\varkappa>0$,
	then there exists a constant $V_\infty$ such that for any $|\alpha|$, $j$, $\ell \ge 0$,
	\begin{equation}\label{est_DVDio}
		|N^T\nabla_\theta \partial_\theta^\alpha \partial_t^j V|+ |\partial_\theta^\alpha \partial_t^{1+j} V| + \varkappa|\partial_\theta^\alpha (V - V_\infty)| \le \frac{ C_\ell \norm{\phi}_{C^k(\T^d)}}{(1+\varkappa t)^\ell},
	\end{equation}
	where $k = k(|\alpha|,j,\ell,d)$ and  $C_\ell$ depends only on $d$, $m$, $|\alpha|$, $j$, $\ell$ and $A$.
\end{lemma}

%\begin{proof}
%	The lemma was proved, for example, in \cite{GerMas12}. Note that $\partial_t V$ is smooth due to the Sobolev embedding theorem and (\ref{est_DDrV}). Thus, we can obtain a pointwise formula for $V$ itself:
%	\begin{equation}\label{key}
%	V(\theta,t) = \int_0^t \partial_\rho V(\theta,\rho) d\rho + f(\theta),
%	\end{equation}
%	which implies the smoothness of $V$.
%\end{proof}

\begin{remark}\label{rmk_coinc}
	The solution of (\ref{eq_Hu}) given by Theorem \ref{thm_Poi}  coincides 
	with the solution of (\ref{eq_Hu}) given by Lemma \ref{lemma-DKV} via (\ref{eq_uV0}) for any $n\in \bS^{d-1}$. 
	To see this,
	let  $w(x) = u^s(x) - V(x-(x\cdot n)n-sn, -x\cdot n -s)$. Clearly, $w$ satisfies
	\begin{equation}
		\left\{
		\begin{aligned}
			\cL^*_1 w &= 0 &\quad & \txt{in } \bH^d_n (s), \\
			w &= 0 &\quad & \txt{on } \partial\bH^d_n (s).
		\end{aligned}
		\right.
	\end{equation}
	Since $u^s$ is bounded and $V$ satisfies
	\begin{align*}
		|V(\theta, t)| & = \bigg| \int_0^t \partial_\rho V(\theta,\rho) d\rho + \phi(\theta) \bigg| \\
		& \le \norm{\phi}_\infty + t^{1/2} \bigg( \int_0^{\infty} |\partial_\rho
		 V(\theta,\rho)|^2 \,d\rho \bigg)^{1/2} \\
		& \le \norm{\phi}_\infty + C t^{1/2} \bigg( \int_0^\infty \norm{\partial_\rho 
		V(\cdot,\rho)}_{H^k(\T^d)}^2 \, d\rho \bigg)^{1/2} \\
		& \le \norm{\phi}_\infty + Ct^{1/2} \| f\|_{H^{k+2}(\T^d)},
	\end{align*}
for some $k\ge 1$,
	we conclude that $w$ is of sublinear growth as $|x\cdot n| \to \infty$. Thus, by Remark \ref{remark-20}, 
	we obtain $w \equiv 0$.
\end{remark}

Now we give an explicit expression for $\overline{f}(x) $ if $n(x)\in \bS^{d-1}_D$. 
For $1\le k\le d$ and $1\le \beta\le m$,
 let $V^{\beta}_{n,k} = V^{\beta}_{n,k}(\theta,t)$ denote the solution of
 the following Dirichlet problem,
 
\begin{equation}\label{eq_Vn}
	\left\{
	\begin{aligned}
		-\Dr{N^T \nabla_\theta}{\partial_t}
		\cdot B_n \Dr{N^T \nabla_\theta}{\partial_t} V^{\beta}_{n,k}
		&= 0 &\quad & \txt{ in } \T^d\times (0,\infty), \\
		V^{\beta}_{n,k} &= -\chi_k^{*\beta} &\quad & \txt{ on }  \T^d\times \{0\},
	\end{aligned}
	\right.
\end{equation}
where $\chi_k^{*\beta}$ are the correctors for $\cL_\e^*$,
 $B_n = M^T A^*(\theta-tn)M$,
  and $M = (N,-n)$ is an orthogonal matrix. 

\begin{theorem}\label{thm_mu}
	Let $x\in \partial\Omega$.
	Suppose that $n=n(x)\in \bS^{d-1}_D$.
	Let $V_n(\theta,t)$ be the solution of (\ref{eq_Vn}). Then
	\begin{equation}\label{eq_mu_barf}
	\overline{f}^\alpha (x)
	 = \int_{\T^d} h^{\alpha\beta} \bigg[\delta^{\gamma\beta} 
	 + \frac{\partial}{\partial \theta_\ell} \chi^{*\gamma\beta}_k(\theta) n_\ell n_k - {\partial_t} V^{\gamma\beta}_{n,k}(\theta,0) \cdot n_k\bigg] a_{ij}^{\gamma \nu}(\theta) n_i n_j f^\nu(x,\theta)\, d\theta
	\end{equation}
	for $1\le \alpha\le m$,
	where $h=(h^{\alpha\beta})$
	denotes the inverse matrix of the $m\times m$ matrix $ (\widehat{a}^{*\alpha\beta}_{ij} n_i n_j)$.
\end{theorem}

\begin{proof}
   This was  proved in \cite{AKMP16} (also see \cite{Z17}).
   \end{proof}
   
We now turn to the proof of Theorem \ref{main-thm-01}. 
The key step is to prove the following.

\begin{theorem}\label{thm_MU}
Fix $\sigma \in (0, 1)$.
Let $x, y\in \partial\Omega$ and $|x-y|\le c_0$.
Suppose that $n(x)$, $n(y)\in\bS^{d-1}_D$.
Then  
	\begin{equation}\label{est_mu12}
		| \overline{f}(x) - \overline{f}(y)| \le C_\sigma \varkappa^{-\sigma} |x-y| 
		\left(\int_{\T^d} \| f(\cdot, y)\|_{C^1(\partial\Omega)}^2\, dy \right)^{1/2},
	\end{equation}
	where $\varkappa = \max \big\{\varkappa(n(x)), \varkappa(n(y)) \big\}$ and 
	$C_\sigma$ depends only on $d$, $m$, $\sigma$, $\lambda$, and $\| A\|_{C^k(\T^d)}$ for some
	$k=k(d, \sigma)\ge 1$.
\end{theorem}

To prove Theorem \ref{thm_MU}, in view of the formula (\ref{eq_mu_barf}),
we investigate the continuity in $n$ of the solution to the Dirichlet problem (\ref{eq_Vs}).

\begin{lemma}\label{lem_V}
For $\phi\in C^\infty(\T^d; \R^m)$,
	let $V$ be the solution of (\ref{eq_Vs}), given by Lemma  \ref{lemma-DKV}, 
 with  $n\in \bS^{d-1}$. Then
	\begin{equation}\label{est_DrV}
		|N^T\nabla_\theta V| + |\partial_t V| \le \frac{C \norm{\phi}_{C^2(\T^d)}}{1+t},
	\end{equation}
	where $C$ depends only on $d$, $m$ and $A$. 
	Moreover, for any $|\alpha|$, $j\ge 0$ and $0<\sigma<1$,
	\begin{equation}\label{est_DrdkV}
		|N^T\nabla_\theta \partial_\theta^\alpha \partial_t^j V| + |\partial_\theta^\alpha \partial_t^{1+j} V| \le
		\frac{ C_\sigma \norm{\phi}_{C^k(\T^d)}}{(1+t)^{1-\sigma}},
	\end{equation}
	where  $k = k(|\alpha|,j,\sigma,d)$ and $C_\sigma$ depends only on $d$, $m$,
	$|\alpha|$, $j$, $\sigma$ and $A$.
\end{lemma}

\begin{proof}
	Let $u^s$ be given by (\ref{eq_uV0}).
	Then
	\begin{equation}\label{eq_us}
		\left\{
		\begin{aligned}
			\cL^*_1 u^s &= 0 &\quad & \txt{in } \bH^d_n (s), \\
			u^s &= \phi &\quad & \txt{on } \partial\bH^d_n(s).
		\end{aligned}
		\right.
	\end{equation}
	It follows from  (\ref{eq_uV0}) that
	\begin{equation}\label{eq_Vthet}
		V(\theta,t) = u^{-\theta \cdot n}(\theta - tn) \quad \txt{for all } (\theta,t) \in \T^d\times \R_+.
	\end{equation}
	Thanks to the fact $N^T\nabla_\theta (\theta\cdot n) = 0$, the last equality implies that
	\begin{equation}\label{eq_Vtt}
		\left\{
		\begin{aligned}
			N^T \nabla_\theta V(\theta,t) &= N^T \nabla_x u^{-\theta \cdot n}(\theta - tn), \\
			\partial_t V(\theta,t) &= -n\cdot \nabla_x u^{-\theta \cdot n}(\theta - tn).
		\end{aligned}
		\right.
	\end{equation}
	As a result, estimates for $N^T \nabla_\theta V$ and 
	$\partial_t V$ may be reduced to the corresponding estimates for $u^s$.
	
	It follows from the presentation of Poisson integral (\ref{eq_Pint}) and the pointwise estimate (\ref{est_DPoison}) that
	\begin{equation}\label{est_Dus}
		|\nabla u^s(x)| \le \frac{C \norm{\phi}_{\infty}}{|s+x\cdot n|}.
	\end{equation}
	To deal with the  case where
	$|s+x\cdot n| = \txt{dist}(x,\partial\bH^d_n(s))<1$, we first
	note that $\norm{u^s}_\infty \le C \norm{\phi}_\infty$ by (\ref{max-p}). 
	Next, by the boundary Lipschitz estimate, we obtain $|\nabla u^s(x)| \le C \norm{\phi}_{C^2(\T^d)}$ 
	if $\txt{dist}(x,\partial\bH^d_n (s))<1$. This, together with (\ref{est_Dus}) and (\ref{eq_Vtt}), proves (\ref{est_DrV}). 
	
	Finally, we prove the inequality (\ref{est_DrdkV}) 
	by using interpolation and the Sobolev embedding. 
	Precisely, for any $L>0$, it follows from (\ref{est_DrV}), (\ref{est_DDrV}) and interpolation that
	\begin{align*}
		\norm{N^T \nabla_\theta V}_{H^{r+\frac{d}{2}+1}(\T^d\times [L,L+1])} 
		& \le C \norm{N^T \nabla_\theta V}_{L^2(\T^d \times [L,L+1])}^{1-\sigma}
		 \norm{N^T \nabla_\theta V}_{H^{k-1}(\T^d\times [L,L+1])}^{\sigma} \\
		& \le C (1+L)^{-(1-\sigma)} \norm{\phi}_{C^{k}(\T^d)},
	\end{align*}
	where $k=k(d, r, \sigma)\ge 1$ is sufficiently large. 
	It follows from the Sobolev embedding theorem that
	\begin{align*}
		\sup_{(\theta,t)  \in \T^d\times [L,L+1]} |N^T \nabla_\theta \partial_\theta^\alpha \partial_t^j V(\theta,t)| 
		& \le C \norm{N^T \nabla_\theta V}_{H^{r+d/2+1}(\T^d\times [L,L+1])} \\
		& \le \frac{ C\norm{\phi}_{C^{k}(\T^d)}}{(1+L)^{1-\sigma}},
	\end{align*}
	which readily implies
	\begin{equation}\label{key}
		|N^T \nabla_\theta \partial_\theta^\alpha \partial_t^j V(\theta,t)| \le \frac{C\norm{\phi}_{C^k(\T^d)}}{(1+t)^{1-\sigma}} \qquad \txt{for any } (\theta,t)\in \T^d\times \R_+,
	\end{equation}
	where $|\alpha|+j\le r$. 
	A similar argument gives the pointwise estimate for $|\partial_\theta^\alpha \partial_t^{1+j} V|$.
\end{proof}

\begin{lemma}
	Let $V$ be the solution of (\ref{eq_Vs}) with $n\in \bS^{d-1}$. Then
	\begin{equation}\label{est_Vbd}
		|V| + |\nabla_\theta V| \le C \norm{\phi}_{C^2(\T^d)},
	\end{equation}
	where $C$ depends only on $d$, $m$ and $A$. 
	Moreover, if $n\in \bS^{d-1}_D$ with Diophantine constant $\varkappa = \varkappa(n)>0$,
	 then for any $|\alpha|\ge 2$ and $0<\sigma<1$,
	\begin{equation}\label{est_dkVbd}
		|\partial_\theta^\alpha V| \le C \varkappa^{-\sigma} \norm{\phi}_{C^k(\T^d)},
	\end{equation}
	where $k = k(d,|\alpha|,\sigma)>1$ and $C$ depends only on $d$, $m$, $|\alpha|$, $\sigma$ and
	$A$.
\end{lemma}
\begin{proof}
	Again, the desired estimates for $V$ will be reduced to estimates for
	 solutions $u^s$ of (\ref{eq_us}), where $V$ and $u^s$ are related by (\ref{eq_Vthet}). 
	 First, since $\norm{u^s}_\infty \le C \norm{\phi}_\infty$,
	 we obtain $|V| \le C \norm{\phi}_\infty$. 
	 Next, by comparing $u^s$ and $u^{s'}$ in the common domain, we 
	 may deduce from the boundary Lipschitz estimate and the Miranda-Agman maximal principle 
	 (\ref{max-p}) that 
	 \begin{equation}\label{e-10}
	 |u^s(x) - u^{s'}(x)|  \le  C |s-s'| \norm{\phi}_{C^2(\T^d)},
	 \end{equation}
	  if $ x\cdot n < -\max\{s,s'\} $.
	Observe that, to prove the boundedness of $\nabla_\theta V$, it suffices to prove the boundedness of $n\cdot \nabla_\theta V$, as $N^T \nabla_\theta V$ is  bounded due to (\ref{est_DrV}). To this end, note that
	\begin{align*}
		&|V(\theta+rn,t) - V(\theta,t)| = |u^{-\theta \cdot n-r}(\theta+rn - tn) - u^{-\theta \cdot n}(\theta - tn)| \\
		&  \le |u^{-\theta \cdot n-r}(\theta+rn - tn) - u^{-\theta \cdot n-r}(\theta - tn)| + |u^{-\theta \cdot n-r}(\theta - tn) - u^{-\theta \cdot n}(\theta - tn)| \\
		&  \le |r| \norm{\nabla u^{-\theta \cdot n-r}}_\infty + \norm{u^{-\theta\cdot n-r} - u^{-\theta\cdot n}}_{\infty} \\
		&  \le C |r| \norm{\phi}_{C^2(\T^d)},
	\end{align*}
	where we have used (\ref{e-10}) for the last step.
	Dividing by $r$ on both sides and taking the limit as $r\to 0$, we 
	obtain $|n\cdot \nabla_\theta V| \le C \norm{\phi}_{C^2(\T^d)}$. This finishes the proof of (\ref{est_Vbd}).
	
	Finally, to show (\ref{est_dkVbd}), we use (\ref{est_Vbd}), (\ref{est_DVDio}) and an interpolation argument. Precisely, let $L>0$ and $t\in [L,L+1]$,
	\begin{align*}
		\sup_{(\theta,t)  \in \T^d\times [L,L+1]} |\partial_\theta^\alpha V(\theta,t)| 
		&\le C  \norm{V}_{H^{d/2+|\alpha|+1}(\T^d\times [L,L+1])} \\
		& \le C \norm{V}_{H^1(\T^d\times [L,L+1])}^{1-\sigma} \norm{V}_{H^r(\T^d \times [L,L+1])}^{\sigma} \\
		& \le C  \varkappa^{-\sigma} \norm{\phi}_{C^k(\T^d)},
	\end{align*}
	where $|\alpha|\ge 2$ and $r=r(d, \alpha, \sigma)$, $k = k(d,|\alpha|,\sigma)$ are 
	sufficiently large. 
	The desired estimate follows.
\end{proof}

Now we are ready to prove Theorem \ref{thm_MU}.

\begin{proof}[\bf Proof of Theorem \ref{thm_MU}]

	{\bf Step 1: Set-up and reduction.}
	
	\medskip
	
	Fix $n_1, n_2\in \bS^{d-1}_D$.
	We may assume that $\delta = |n_1 - n_2|>0$ is sufficiently small. 
	Let $N_1$ and $N_2$ be the $d\times (d-1)$ matrices such that both $M_1 = (N_1,-n_1)$ and $M_2 = (N_2,-n_2)$ are orthogonal matrices. Recall that solution $V_1$ (resp. $V_2$) of (\ref{eq_Vs}), associated with $n_1$ (resp. $n_2$), is independent of the choices of $N_1$ (resp. $N_2$). So without loss of generality, we 
	may assume $|N_1-N_2|\le C \delta$. To be precise, we write down the systems 
	for $V_1$ and $V_2$ as follows:
	\begin{equation}\label{eq_V1}
		\left\{
		\begin{aligned}
			- \Dr{N^T_1 \nabla_\theta}{\partial_t} \cdot B_1
			\Dr{N^T_1 \nabla_\theta}{\partial_t} V_1
			&= 0  &\quad & \txt{ in } \T^d\times (0,\infty), \\
			V_1 &= \phi  &\quad & \txt{ on }  \T^d\times \{0\},
		\end{aligned}
		\right.
	\end{equation}
	and
	\begin{equation}
		\left\{
		\begin{aligned}
			- \Dr{N^T_2 \nabla_\theta}{\partial_t} \cdot B_2
			\Dr{N^T_2 \nabla_\theta}{\partial_t} V_2
			&= 0  &\quad & \txt{ in } \T^d\times (0,\infty), \\
			V_2 &= \phi  &\quad & \txt{ on }  \T^d\times \{0\},
		\end{aligned}
		\right.
	\end{equation}
	where $B_\ell (\theta,t) = M_\ell ^T A^*(\theta-tn_\ell)M_\ell$ for $\ell=1,2$
	and $\phi=-\chi_k^{*\beta}$.
	In view of Theorem \ref{thm_mu}, to show (\ref{est_mu12}), it suffices to prove that
	\begin{equation}\label{est_dtV12}
		\int_{\T^d}|\partial_t V_1(\theta,0) - \partial_t V_2(\theta,0)|^2\,
		 d\theta \le C \varkappa^{-2\sigma}|n_1 - n_2|^2 .
	\end{equation}
	Define $W = V_1 - V_2$. Observe that
	\begin{equation}\label{est_t01}
		\int_{\T^d} |\partial_t W(\theta,0)|^2 \, d\theta \le 2\int_0^1\int_{\T^d} |\partial_t W(\theta,t)|^2 \, d\theta dt + 2\int_0^1\int_{\T^d} |\partial_t^2 W(\theta,t)|^2 \, d\theta dt.
	\end{equation}
	Thus, the estimate (\ref{est_dtV12}) is further reduced to that for the two integrals 
	in the RHS of (\ref{est_t01}).
	We may assume that $\varkappa(n_1)\ge \varkappa(n_2)$ and thus $\varkappa = \varkappa(n_1)$. 
	
	\medskip
	
	{\bf Step 2: Estimate for $\partial_t W$.} 
	
	\medskip
	
	Note that $W$ satisfies $W(\theta,0) = 0$ and
	\begin{align}\label{eq_tW}
		\begin{aligned}
			& -  \Dr{N^T_2 \nabla_\theta}{\partial_t} \cdot B_2
			\Dr{N^T_2 \nabla_\theta}{\partial_t} W\\
			&=
			- \Dr{N^T_2 \nabla_\theta}{\partial_t} \cdot B_2
			\Dr{N^T_2 \nabla_\theta}{\partial_t}  V_1 \\
			& = \bigg[ \Dr{N^T_1 \nabla_\theta}{\partial_t} \cdot B_1
			\Dr{N^T_1 \nabla_\theta}{\partial_t} - \Dr{N^T_2 \nabla_\theta}{\partial_t} \cdot B_2
			\Dr{N^T_2 \nabla_\theta}{\partial_t} \bigg] V_1.
		\end{aligned}
	\end{align}
	By using
	\begin{align*}
		&  \Dr{N^T_1 \nabla_\theta}{\partial_t} \cdot B_1
		\Dr{N^T_1 \nabla_\theta}{\partial_t} - \Dr{N^T_2 \nabla_\theta}{\partial_t} \cdot B_2
		\Dr{N^T_2 \nabla_\theta}{\partial_t} \\
		& = - \Dr{N^T_2 \nabla_\theta}{\partial_t} \cdot B_2
		\Dr{(N^T_2-N^T_1) \nabla_\theta}{0}
		- \Dr{N^T_2 \nabla_\theta}{\partial_t} \cdot (B_2-B_1)
		\Dr{N^T_1 \nabla_\theta}{\partial_t} \\
		&\qquad\qquad +  \Dr{(N^T_2-N^T_1) \nabla_\theta}{0} \cdot B_1
		\Dr{N^T_1 \nabla_\theta}{\partial_t},
	\end{align*}
	the RHS of (\ref{eq_tW}) can be written as
	\begin{equation}\label{eq_G12H}
		\Dr{N^T_2 \nabla_\theta}{\partial_t} \cdot (G_1 + G_2) + H,
	\end{equation}
	where
	\begin{align*}
		G_1 & = -B_2 \Dr{(N^T_2-N^T_1) \nabla_\theta}{0} V_1,\\
		G_2 & = -(B_2-B_1)
		\Dr{N^T_1 \nabla_\theta}{\partial_t} V_1,\\
		H & =  \Dr{(N^T_2-N^T_1) \nabla_\theta}{0} \cdot B_1
		\Dr{N^T_1 \nabla_\theta}{\partial_t} V_1.
	\end{align*}
	Therefore, the equation (\ref{eq_tW}) is reduced to
	\begin{equation}\label{eq_tWGH}
		- \Dr{N^T_2 \nabla_\theta}{\partial_t} \cdot B_2
		\Dr{N^T_2 \nabla_\theta}{\partial_t} W = \Dr{N^T_2 \nabla_\theta}{\partial_t}\cdot G + H,
	\end{equation}
	where $G = G_1 + G_2$.
	
	\begin{lemma}\cite[Remark 6.3]{ShenZhuge16}\label{lem_A1est}
		Let $n\in \bS^{d-1}$  and $U$ be a smooth solution of
		\begin{equation}
			- \Dr{N^T \nabla_\theta}{\partial_t} \cdot B
			\Dr{N^T \nabla_\theta}{\partial_t} U = \Dr{N^T \nabla_\theta}{\partial_t} \cdot G + H,
		\end{equation}
		with $U(\cdot,0) = 0$. Assume that
		\begin{equation}
			\sup_{t>0,\theta\in \T^d} (1+t)\Big\{ |N^T\nabla_\theta U(\theta,t)| + |\partial_t U(\theta,t)| + |G(\theta,t)| + (1+t)|H(\theta,t)| \Big\} < \infty.
		\end{equation}
		Then, for any $0<\sigma<1$,
		\begin{equation}
			\int_0^\infty \int_{\T^d} \left(|N^T \nabla_\theta U|^2 + 
			|\partial_t U|^2\right) t^{\sigma-1}\,  d\theta dt \le
			C\int_0^\infty \int_{\T^d} \left(|G|^2 + t^2 |H|^2 \right) t^{\sigma-1} \, d\theta dt.
		\end{equation}
	\end{lemma}
	
	Applying Lemma \ref{lem_A1est} to the system (\ref{eq_tWGH}), we obtain
	\begin{equation}\label{est_DrW_GH}
	\aligned
		\int_0^\infty \int_{\T^d}  &\left(|N^T_2 \nabla_\theta W|^2 + |\partial_t W|^2\right) t^{\sigma-1} \, d\theta dt \\
		 & \le C  \int_0^\infty \int_{\T^d} \left(|G|^2 +  t^2|H|^2\right) t^{\sigma-1} \, d\theta dt.
		\endaligned
	\end{equation}
	Hence, it suffices to estimate the integrals involving $G$ and $H$ in (\ref{est_DrW_GH}).
	
	{\bf Estimate for the integral with $G_1$:} By the estimates for $|\nabla V_1|$ in (\ref{est_Vbd}) and (\ref{est_DVDio}), we have
	\begin{equation}\label{key}
		|G_1(\theta,t)| \le C \delta |\nabla_\theta  V_1(\theta,t)| \le C \delta \cdot 1^{1-\sigma}
		[\varkappa^{-1} (1+\varkappa t)^{-\ell}]^{\sigma}
	\end{equation}
	for any $0<\sigma<1$. It follows that
	\begin{align*}\label{key}
		\int_0^\infty \int_{\T^d} |G_1|^2 t^{\sigma-1} \,
		d\theta dt & \le C  \delta^2 \varkappa^{-2\sigma} \int_0^\infty \frac{dt}{t^{1-\sigma} (1+\varkappa t)^{2\ell\sigma}} \\
		& \le C\delta^2 \varkappa^{-3\sigma} \int_0^\infty \frac{dt}{t^{1-\sigma} (1+t)^{2\ell\sigma}}  \\
		& \le C \delta^2 \varkappa^{-3\sigma},
	\end{align*}
	where we can simply choose $\ell = 1$ to ensure the convergence of the integral in the right-hand side.
	
	{\bf Estimate for the integral with $G_2$:} 
	Note that an interpolation between (\ref{est_DrV}) and (\ref{est_DVDio}) implies
	\begin{equation}\label{est_V1_itp}
		|N^T_1 \nabla_\theta V_1(\theta,t)| + |\partial_t V_1(\theta,t)| \le
		C (1+t)^{\sigma-1} (1+\varkappa t)^{-\ell\sigma}.
	\end{equation}
	Also note that $|B_1(\theta,t)-B_2(\theta,t)| \le C t \delta$. It follows that
	\begin{align*}\label{key}
		\int_0^\infty \int_{\T^d} |G_2|^2 t^{\sigma-1} \, d\theta dt & \le C \delta^2 \int_0^\infty \frac{t^{1+\sigma} dt}{(1+t)^{2(1-\sigma)} (1+\varkappa t)^{2\ell\sigma}}  \\
		& \le C \delta^2 \varkappa^{-3\sigma},
	\end{align*}
	where we need to choose $\ell = 2$.
	
	{\bf Estimate for the integral with $H$:} Observe that
	\begin{align*}
		\int_0^\infty \int_{\T^d} |H|^2 t^{1+\sigma} \,
		d\theta dt & \le C  \delta^2 \int_0^\infty \int_{\T^d} 
		\left(|N^T_1\nabla_\theta V_1|^2 + |\partial_t V_1|^2\right) t^{1+\sigma}\, d\theta dt \\
		& \qquad + C \delta^2 \int_0^\infty \int_{\T^d} 
		\left(|N^T_1\nabla_\theta \nabla_\theta V_1|^2 + |\partial_t \nabla_\theta V_1|^2\right) 
		t^{1+\sigma}\, d\theta dt .
	\end{align*}
	The first term in the RHS is bounded by $\delta^2 \varkappa^{-3\sigma}$ by using (\ref{est_V1_itp}). 
	To handle the second integral, we apply the interpolation theorem between (\ref{est_DrdkV}) and (\ref{est_DVDio}) to obtain
	\begin{equation}\label{key}
		|N^T_1 \nabla_\theta \nabla_\theta V_1(\theta,t)|
		 + |\partial_t \nabla_\theta V_1(\theta,t)| \le C (1+t)^{-(1-\sigma)^2} (1+\varkappa t)^{-\ell\sigma}.
	\end{equation}
	Thus, the second term is bounded by
	\begin{equation}\label{key}
		C \delta^2 \int_0^\infty \frac{t^{1+\sigma} \, dt}{(1+t)^{2(1-2\sigma)} (1+\varkappa t)^{2\ell\sigma}}  \le
		C \delta^2 \varkappa^{-5\sigma},
	\end{equation}
	where we have chosen $\ell = 3$.
	
	By combining the estimates above with (\ref{est_DrW_GH}), we obtain 
	\begin{equation}\label{est_dW01}
		\int_0^1 \int_{\T^d} \left(|N^T_2 \nabla_\theta W|^2 
		+ |\partial_t W|^2\right)  d\theta dt \le C_\sigma \delta^2 \varkappa^{-5\sigma}.
	\end{equation}
	
	{\bf Step 3: Estimate for $\partial_t^2 W$.}
	
	Let $N_{2j}$ denote the $j$th column of $N_2$ and define $\nabla_{2j} = N_{2j}^T\cdot \nabla_\theta$
	for $1\le j\le d-1$. Note that $\nabla_{2j}$ is the $j$th component of $N_2^T\nabla_\theta$. Then we apply $\nabla_{2j}$ to (\ref{eq_tWGH}) and obtain
	\begin{align}\label{eq_djW}
		\begin{aligned}
			- \Dr{N^T_2 \nabla_\theta}{\partial_t} \cdot B_2
			\Dr{N^T_2 \nabla_\theta}{\partial_t} \nabla_{2j}W
			&= \Dr{N^T_2 \nabla_\theta}{\partial_t} \cdot \nabla_{2j}G + \nabla_{2j}H \\
			&\qquad +\Dr{N^T_2 \nabla_\theta}{\partial_t} \cdot \nabla_{2j}B_2
			\Dr{N^T_2 \nabla_\theta}{\partial_t} W，
		\end{aligned}
	\end{align}
	on $\T^d\times \R_+$ and $\nabla_{2j}W = 0$ on $\T^d\times \{0\}$. Let $\eta(t)$ 
	be a cut-off function such that $\eta(t) = 1$ for $t\in [0,1]$,
	$\eta(t) = 0$ for $t\in [2,\infty)$, 
	$0\le \eta(t)\le 1$ and $|\nabla \eta|\le C$. Now integrating (\ref{eq_djW}) against $ \eta^2\nabla_{2j}W$, we derive from integration by parts that
	\begin{align*}
		&\int_0^1 \int_{\T^d} \left(|N^T_2 \nabla_\theta \nabla_{2j} W|^2 
		+ |\partial_t \nabla_{2j} W|^2\right) d\theta dt \\
		&\qquad \le C \int_0^2 \int_{\T^d}
		\left(|\nabla_{2j}G|^2 + |\nabla_{2j}H|^2 +|N_2^T \nabla_\theta W|^2 +|\partial_t W|^2\right) d\theta dt\\
		&\qquad \le C \varkappa^{-5\sigma} \delta^2,
	\end{align*}
	where we have used the fact $|\nabla_{2j}W| \le |N_2^T\nabla_\theta W|$. 
	Consequently,
	\begin{equation}\label{est_ddW01}
		\int_0^1 \int_{\T^d} \left(|N^T_2 \nabla_\theta \otimes N^T_2 \nabla_\theta W |^2 
		+ |\partial_t N^T_2 \nabla_\theta W|^2\right) d\theta dt \le C \varkappa^{-5\sigma} \delta^2.
	\end{equation}
	
	Now observe that
	\begin{align*}
		&\Dr{N^T_2 \nabla_\theta}{\partial_t} \cdot B_2
		\Dr{N^T_2 \nabla_\theta}{\partial_t}W \\
		&= \Dr{N^T_2 \nabla_\theta}{\partial_t} B_2 \cdot
		\Dr{N^T_2 \nabla_\theta}{\partial_t}W + B_2: \bigg[\Dr{N^T_2 \nabla_\theta}{\partial_t} \otimes \Dr{N^T_2 \nabla_\theta}{\partial_t}\bigg] W \\
		& = \Dr{N^T_2 \nabla_\theta}{\partial_t} B_2 \cdot
		\Dr{N^T_2 \nabla_\theta}{\partial_t}W + B_2: \bigg[\begin{matrix}
			N^T_2\nabla_\theta \otimes N^T_2\nabla_\theta & N^T_2\nabla_\theta \partial_t \\
			(N^T_2\nabla_\theta \partial_t)^T & 0
		\end{matrix} \bigg] W + b_{2,dd} \partial_t^2 W,
	\end{align*}
	where $b_{2,dd} = (b_{2,dd}^{\alpha\beta})_{1\le \alpha,\beta\le m}$ is positive due to the strong ellipticity condition. This gives
	\begin{align*}
		b_{2,dd} \partial_t^2 W &= -\Dr{N^T_2 \nabla_\theta}{\partial_t} B_2 \cdot
		\Dr{N^T_2 \nabla_\theta}{\partial_t}W - B_2: \bigg[\begin{matrix}
			N^T_2\nabla_\theta \otimes N^T_2\nabla_\theta & N^T_2\nabla_\theta \partial_t \\
			(N^T_2\nabla_\theta \partial_t)^T & 0
		\end{matrix} \bigg] W \\
		& \qquad \qquad- \Dr{N^T_2 \nabla_\theta}{\partial_t}\cdot G - H.
	\end{align*}
	Note that $|(b_{2,dd})^{-1}| \le C$. Thus, it follows from (\ref{est_dW01}), (\ref{est_ddW01}) and the pointwise estimates of $G$ and $H$ for $t\in [0,1]$ that
	\begin{equation}\label{key}
		\int_0^1 \int_{\T^d} |\partial_t^2 W|^2 \, d\theta dt \le C \delta^2 \varkappa^{-5\sigma}.
	\end{equation}
	This completes the proof of  Theorem \ref{thm_MU}.
\end{proof}

\begin{proof}[\bf Proof of Theorem \ref{main-thm-01}]

	Note that $\partial\Omega$ is locally differential homeomorphic to $\R^{d-1}$. Thus, in view of Theorem \ref{thm_MU}, it suffices to prove the following claim: 
	Let $F\in L^1(\R^{d-1};\R^m)$ and $G\in L^p(\R^{d-1})$ for some $1<p<\infty$.
	Suppose that for  
	a.e.  $x\in \R^{d-1}$,
	\begin{equation}\label{est_fxyg}
		|F(x) - F(y)| \le  |x-y| |G(x)|, \qquad \txt{for a.e. } y \in \R^{d-1}.
	\end{equation} 
	Then
	\begin{equation}\label{est_dfg}
		\left(\int_{\R^{d-1}} |\nabla F| ^p\right)^{1/p}
		 \le C \left(\int_{\R^{d-1}} |G|^p\right)^{1/p},
	\end{equation}
	where $C$ depends only on $d$ and $p$.
	Indeed, if the claim holds, then it follows from Theorem \ref{thm_MU} that
	\begin{equation}\label{key}
		\left(\int_{\partial\Omega} |\nabla_{\txt{tan}} \overline{f} |^p\right)^{1/p}
		 \le C   \left(\int_{\T^d} \| f (\cdot, y)\|^2_{C^1(\partial\Omega)}
		 d y\right)^{1/2}
		 \left(\int_{\partial\Omega} \left[ \varkappa (n(x) )\right]^{-\sigma p}\, dx \right)^{1/p}
	\end{equation}
	for any $0<\sigma<1$. Recall that $[\varkappa(n(x))]^{-1} \in L^q(\partial\Omega)$ for any $q<d-1$. Thus, for any $p<\infty$, we choose $\sigma\in (0,1)$ so small that $\sigma p<d-1$. As a result, we obtain 
	\begin{equation}\label{key}
		\left(\int_{\partial\Omega} |\nabla_{\txt{tan}} \overline{ f}|^p\right)^{1/p}
		 \le C  \left(\int_{\T^d} \| f (\cdot, y)\|^2_{C^1(\partial\Omega)}
		 d y\right)^{1/2}
	\end{equation}
	for any $p<\infty$. Note that $\overline{f}$ is bounded. 
	We may conclude that $\overline{f} \in W^{1,p}(\partial\Omega;\R^m)$ and (\ref{main-estimate-01}) holds.
	
	It remains to prove the claim. Let $\varphi\in C_0^\infty(B(0,1))$ and $\int_{\R^{d-1}} 
	\varphi = 1$. 
	Set $\varphi_\e(x) = \e^{1-d} \varphi(x/\e)$. Define for any $\e > 0$,
	\begin{equation}\label{key}
		F_\e(x) = \int_{\R^{d-1}} F(y) \varphi_\e(x-y) \, dy.
	\end{equation}
	Clearly, $F_\e$ is smooth and $F_\e \to F$ 
	in $L^1(\R^{d-1};\R^m)$ as $\e \to 0$. 
	Moreover, for any $z\in B(x,\e)$,
	\begin{align*}\label{key}
		\nabla F_\e(x) & = \int_{\R^{d-1}} F(y) \nabla \varphi_\e(x-y) \, dy \\
		& = \int_{\R^{d-1}} ( F(y) - F(z)) \nabla \varphi_\e(x-y) \, dy.
	\end{align*}
	Using the assumption (\ref{est_fxyg}),
	\begin{align*}
		|\nabla F_\e(x)| & \le  \fint_{ B(x,\e)}  |G(z)| \int_{B(x,\e)} |y-z| |\nabla \varphi_\e(x-y)| \, dydz \\
		& \le C  \fint_{B(x,\e)} |G(z)| \, dz\\
		&\le C  \left(\fint_{B(x, \e)} |G(z)|^p\, dz\right)^{1/p}.
	\end{align*}
	Thus, by Fubini's Theorem, for any $\e>0$
	\begin{equation}\label{key-100}
		\left(\int_{\R^{d-1}} |\nabla F_\e(x)|^p \, dx\right)^{1/p} 
		\le C  \left(\int_{\R^{d-1}} |G(z)|^p \, dz\right)^{1/p}.
	\end{equation}
	Since $\nabla F_\e \to \nabla F$ in the sense of distribution as $\e\to 0$,
	 (\ref{est_dfg}) follows
	from (\ref{key-100}).
\end{proof}

%%%%%%%%%%%%%%%%%%%%%%%%%%%%%%%%%%%%%%%%%%%%%%%%%

%%%%%%%%%%%%%%%%%%%%%%%%%%%%%%%%%%%%%%%%%%%%%%%%%%

\section{Regularity for Neumann problems}

As in the case of Dirichlet problems, to establish  the regularity of $\overline{g}_{ij}$,
 we use an explicit formula for $\overline{g}_{ij}$ previously discovered in \cite{ShenZhuge16}.
It involves a family of Neumann problems in the half-spaces:

\begin{equation}\label{eq_usN}
	\left\{
	\begin{aligned}
		\cL^*_1 u^s &= 0 &\quad & \txt{in } \bH^d_n(s), \\
		n\cdot A^*\nabla u^s &= T \cdot \nabla \phi &\quad & \txt{on } \partial\bH^d_n(s),
	\end{aligned}
	\right.
\end{equation}
where $T$ is a constant tangential vector, i.e., $T\cdot n = 0$, with $|T|\le 1$.
We assume that $\phi\in C^\infty(\T^d;\R^m)$.

As far as we know, for arbitrary  $n\in \bS^{d-1}$, the solvability of (\ref{eq_usN}) is not clear. 
But for $n\in \bS^{d-1}_D$, it was shown in \cite{ShenZhuge16} 
that (\ref{eq_usN}) is solvable by lifting the problem to a $(d+1)$-dimensional  system
 in the upper half-space, in a manner similar to the case of Dirichlet condition. 
 More precisely, we seek a solution in the form of
\begin{equation}\label{eq_usU}
	u^s(x) = U(x-(x\cdot n+s) n, -(x\cdot n+s)),
\end{equation}
where $U$  is a solution of the Neumann problem:
\begin{equation}\label{eq_U}
	\left\{
	\begin{aligned}
		-\Dr{N^T \nabla_\theta}{\partial_t}
		\cdot B \Dr{N^T \nabla_\theta}{\partial_t} U
		&= 0 &\quad & \txt{ in } \T^d\times \R_+, \\
		-e_{d+1}\cdot B \Dr{N^T \nabla_\theta}{\partial_t} U &= T \cdot \nabla_\theta \phi
		&\quad & \txt{ on }  \T^d\times \{0\},
	\end{aligned}
	\right.
\end{equation}
with $B(\theta,t) = M^T A^*(\theta-tn) M$ and $M = (N,-n)$ being an orthogonal matrix.
%The system (\ref{eq_Un}) is of the same type as (\ref{eq_Vn}) with Neumann boundary data on $\T^d \times \{0 \}$, originating from solving a Neumann problem of first type in a half-space; see [?] for details. The functions $f_{ij,k}^\beta$, completely determined by $A$, are smooth as long as $A$ is. The equations (\ref{eq_fijk}) for $f_{ij,k}^\beta$ are unimportant and will not be used in this paper. To simplify the notation, as in the case of Dirichlet problem, we omit the sub- and superscripts for $U^{*\beta}_{n,k}$ and consider a general version of (\ref{eq_Un}). Precisely, for $n\in \bS^{d-1}$, consider
The solvability of (\ref{eq_U}) and related estimates are addressed below.

\begin{lemma}\cite[Proposition 3.6]{ShenZhuge16}\label{lem_Uexist}
	Suppose that $n$ satisfies the Diophantine condition with constant $\varkappa>0$. Then 
	the Neumann problem (\ref{eq_U}) has a smooth solution $U$, and the solution is unique, up to a constant
	under the condition that $U\in L^\infty(\T^d\times \mathbb{R}_+)$,
	$\nabla_\theta U \in L^2(\T^d\times \mathbb{R}_+)$ and $\partial_t U \in L^2(\T^d\times \mathbb{R}_+)$.
	Moreover, the solution 
	satisfies
	\begin{equation}\label{est_DUSob}
		\int_0^\infty \int_{\T^d} 
		\Big\{
		|N^T\nabla_\theta \partial_\theta^\alpha \partial_t^j U|^2 + |\partial_\theta^\alpha \partial_t^{1+j} U|^2
		\Big\} \, d\theta dt \le C \norm{\phi}^2_{C^{|\alpha|+j+1}(\T^d)},
	\end{equation}
	for any $|\alpha|$, $j \ge 0$, 
	where $C$ depends only on $d$, $m$, $|\alpha|$, $j$, and $A$. 
	Furthermore, there exists a constant vector $U_\infty$ such that for any $|\alpha|$, $j$, $\ell \ge 0$,
	\begin{equation}\label{est_DUptwise}
		|N^T\nabla_\theta \partial_\theta^\alpha \partial_t^j U|+ |\partial_\theta^\alpha \partial_t^{1+j} U| + \varkappa|\partial_\theta^\alpha (U - U_\infty)| \le  \frac{ C_\ell \norm{\phi}_{C^k(\T^d)}}{(1+\varkappa t)^\ell} ,
	\end{equation}
	where $k = k(|\alpha|,j,\ell,d)$ and $C_\ell$ depends only on $d$, $m$, $|\alpha|$, $j$, $\ell$, and $A$.
\end{lemma}

%The decay estimate for (\ref{est_DUDio}) and (\ref{est_UDio}) as $t$ approaching infinity were also obtain in \cite{bibid}. However, in this paper we only need the existence of the solution, as well as the well-definedness of $U$ and $\nabla_\theta U$ in this case. Also, we emphasize that almost every vector in $\bS^{d-1}$ satisfies the Diophantine condition and therefore actually (\ref{eq_U}) is solvable for almost every $n\in \bS^{d-1}$.

%Unlike the Dirichlet problem, the relationship between $U$ and the solutions of half-space problems in not so clear since the solution for a Nuemann problem is not unique. We first consider the system (\ref{eq_U}), of which the general solutions can be given by
%\begin{equation}\label{key}
%U(\theta,t) + \phi(\theta\cdot n), \qquad \txt{as long as } \phi(\theta\cdot n) \txt{ is 1-periodic in } \theta,
%\end{equation}
%where $U$ is one of the solutions of (\ref{eq_U}). This is because for any $\phi:\R\mapsto \R^m$, the function $\phi(\theta \cdot n)$ satisfying the system (\ref{eq_U}) (except for the periodicity) with vanishing Neumann data. Since if $n$ is Diophantine (and therefore irrational), the equidistribution theorem implies that $\phi(\theta\cdot n)$ is 1-periodic if and only if $\phi$ is a constant. We denote this constant by $\phi_0$ and set
%\begin{equation}\label{key}
%\phi_0 = -\int_{\T^d} U(\theta,0) d\theta.
%\end{equation}

\begin{remark}
	Lemma \ref{lem_Uexist} gives the existence of solutions to (\ref{eq_usN}) 
	for  $s\in \R$ and $n\in \bS^{d-1}_D$
	 via (\ref{eq_usU}). Moreover, by the (large-scale) uniform boundary Lispchitz estimates 
	for Neumann conditions \cite{KLS13, AS}, 
	the solution satisfying the sublinear growth as $x\cdot n \to -\infty$ is unique up to a constant.
\end{remark}

Recall that $\bS^{d-1}_D$ has full surface measure of $\bS^{d-1}$. 
An expression for $\overline{g}_{ij}$ 
defined a.e. on $\bS^{d-1}$ is formulated in \cite{ShenZhuge16} and summarized below.

\begin{theorem}\label{thm-N-formula}
	Let  $g=\{ g_{ij}\}$,
	where $g_{ij} \in C^\infty(\partial\Omega\times \T^d; \R^{m}) $. 
	Then, for any $x\in \partial\Omega$ with $n=n(x)\in \bS^{d-1}_D$,
	\begin{equation}\label{eq_tau}
		\overline{g}_{jk}^\gamma(x) 
		= n_i \widehat{a}_{ji}^{\alpha\gamma} h^{\alpha\beta} T_{\ell r} 
		\cdot \int_{\T^d} \Big[ e_k \delta^{\nu\beta} 
		+ \nabla_\theta \chi_k^{*\nu\beta}(\theta) 
		+ \nabla_{\theta} U_{n,k}^{\nu\beta}(\theta,0) \Big] g_{\ell r}^{\nu}(x, \theta) \, d\theta,
	\end{equation}
	where $(h^{\alpha\beta})$ denotes the inverse of
	the $m\times m$ matrix  $(\widehat{a}^{*\alpha\beta}_{ij} n_i n_j)$ and $U_{n,k}^{\beta}$ is the solution of
	\begin{equation}\label{eq_Un}
		\left\{
		\begin{aligned}
			-\Dr{N^T \nabla_\theta}{\partial_t}
			\cdot B_n \Dr{N^T \nabla_\theta}{\partial_t} U^{\beta}_{n,k}
			&= 0 &\quad & \emph{ in } \T^d\times (0,\infty), \\
			-e_{d+1}\cdot B_n \Dr{N^T \nabla_\theta}{\partial_t} U^{\beta}_{n,k} 
			&= \frac{1}{2}T_{ij} \cdot \nabla_\theta \phi_{ij,k}^\beta &\quad & \emph{ on }  \T^d\times \{0\},
		\end{aligned}
		\right.
	\end{equation}
	where $T_{ij} = n_i e_j - n_j e_i$, $B_n (\theta, t)=M^TA^*(\theta-tn)M$, and $\phi_{ij,k}^\beta = 
	(\phi_{ij,k}^{1\beta},\phi_{ij,k}^{2\beta},\cdots,\phi_{ij,k}^{m\beta})$ are the 1-periodic smooth functions satisfying
	\begin{equation}\label{eq_fijk}
		\frac{\partial}{\partial \theta_i}\big\{ \phi^{\alpha\beta}_{ij,k} \big\} = a^{*\alpha\beta}_{jk} + a^{*\alpha\gamma}_{j\ell} \frac{\partial}{\partial \theta_\ell} \chi^{*\gamma\beta}_k - \widehat{a}^{*\alpha\beta}_{jk} \quad \txt{and} \quad \phi^{\alpha\beta}_{ij,k} = - \phi^{\alpha\beta}_{ji,k}.
	\end{equation}
\end{theorem}

We point out that the functions $\phi_{ij,k}^\beta$, which are completely determined by $A$, 
are smooth as long as $A$ is. The equations (\ref{eq_fijk}) for $\phi_{ij,k}^\beta$  
 will not be used in this paper.

\begin{theorem}\label{thm_tauC}
Fix $\sigma \in (0, 1)$.
Let $x, y\in \partial\Omega$ and $|x-y|\le c_0$.
Suppose that $n(x)$, $n(y)\in\bS^{d-1}_D$.
Then  
	\begin{equation}\label{est_tau12}
		| \overline{g}(x) - \overline{g}(y)| \le C_\sigma \varkappa^{-\sigma} |x-y| 
		\left(\int_{\T^d} \| g(\cdot, y)\|_{C^1(\partial\Omega)}^2\, dy \right)^{1/2},
	\end{equation}
	where $\varkappa = \max \big\{\varkappa(n(x)), \varkappa(n(y)) \big\}$ and 
	$C_\sigma$ depends only on $d$, $m$, $\sigma$, $\lambda$, and $\| A\|_{C^k(\T^d)}$ for some
	$k=k(d, \sigma)\ge 1$.
\end{theorem}

To prove Theorem \ref{thm_tauC}, the following two lemmas will be crucial.

\begin{lemma}\label{lem_U}
	Let $n\in \bS^{d-1}_D$ and $U$ be a solution of (\ref{eq_U}) corresponding to $n$. Then
	\begin{equation}\label{est_DrU}
		|N^t\nabla_\theta U| + |\partial_t U| \le \frac{C \norm{\phi}_{C^k(\T^d)}}{1+t},
	\end{equation}
	where $k>d/2+1$ and $C$ depends only on $d,m$ and $A$. Moreover, for any $0<\sigma<1$,
	\begin{equation}\label{est_DrdkU}
		|N^t\nabla_\theta \partial_\theta^\alpha \partial_t^j U| + |\partial_\theta^\alpha \partial_t^{1+j} U| \le
		 \frac{ C_\sigma \norm{\phi}_{C^k(\T^d)}}{(1+t)^{1-\sigma}},
	\end{equation}
	where $k = k(|\alpha|,j,\sigma,d)$ and $C_\sigma$ depends only on $d$, $m$, $|\alpha|$, $j$, $\sigma$ and $A$.
\end{lemma}

\begin{proof}
	Let $u^s$ be the solution of (\ref{eq_usN}), given by (\ref{eq_usU}). Then it follows from \cite[Theorem 4.1]{ShenZhuge16} that
	\begin{equation}\label{est_dus}
		|\nabla u^s(x)| \le \frac{ C\norm{\phi}_\infty}{|x\cdot n +s|} \quad \txt{ for  } x\cdot n + s < 0.
	\end{equation}
	Observe that (\ref{eq_usU}) is equivalent to $U(\theta,t) = u^{-\theta \cdot n}(\theta -tn)$ for any $(\theta,t) \in \T^d\times \R_+$. 
	It follows that
	\begin{equation}\label{eq_Utt}
		\left\{
		\begin{aligned}
			N^T \nabla_\theta U(\theta,t) = N^T \nabla_x u^{-\theta \cdot n}(\theta - tn), \\
			\partial_t U(\theta,t) = -n\cdot \nabla_x u^{-\theta \cdot n}(\theta - tn).
		\end{aligned}
		\right.
	\end{equation}
	In view of (\ref{est_dus}) and (\ref{eq_Utt}) we obtain 
	\begin{equation}
		|N^t\nabla_\theta U(\theta,t)| + |\partial_t U(\theta,t)| \le \frac{ C\norm{\phi}_{L^\infty}}{t}.
	\end{equation}
	This gives (\ref{est_DrU}) for $t\ge 1/2$.
	The case $t\in [0, 1/2]$ follows from (\ref{est_DUSob}) and the Sobolev embedding theorem in $\T^d \times [0,1]$, which requires $k>d/2+1$. 
	
	Finally, the estimate (\ref{est_DrdkU}) follows from (\ref{est_DrU}), (\ref{est_DUSob}) and an interpolation argument, as in the proof of Lemma \ref{lem_V}.
\end{proof}

\begin{lemma}\label{lem_dtheU}
	Let $n\in \bS^{d-1}_D$ with Diophantine constant $\varkappa>0$ and $U$ be a solution of (\ref{eq_U}) corresponding to $n$. Then there exists a constant vector $U_\infty$ 
	such that for any $0<\sigma<1$ and $|\alpha|\ge 0$
	\begin{equation}\label{est_Ubd}
		|\partial_\theta^{\alpha} (U - U_\infty)| \le C_\sigma \varkappa^{-\sigma} \norm{f}_{C^k(\T^d)}.
	\end{equation}
	where $k = k(\alpha,\sigma,d)$ and $C_\sigma$ depends only on $d$, $m$, $\alpha$, $\sigma$, and $A$.
\end{lemma}

\begin{proof}
	We first observe that it suffices to show 
	$|U - U_\infty| \le C_\sigma \varkappa^{-\sigma} \norm{f}_{C^k(\T^d)}$ for any $0<\sigma<1$. 
	Then the  case $|\alpha|>0$ follows from this and (\ref{est_DUptwise}) by an interpolation argument.
	
	Note that $|U-U_\infty| \to 0$ as $t\to \infty$. 
	It follows from (\ref{est_DUptwise}) and (\ref{est_DrU}) that
	\begin{equation}\label{key}
		|\partial_t U(\theta,t)| \le C\,  \frac{\norm{f}_{C^k(\T^d)}^{1-\sigma}}{(1+t)^{1-\sigma}} \cdot \frac{\norm{f}_{C^k(\T^d)}^\sigma}{ (1+\varkappa t)^{\sigma \ell}}.
	\end{equation}
	Hence,
	\begin{align*}
		\sup_{t>0} |(U - U_\infty)(\theta,t)| &\le \int_0^\infty |\partial_t U(\theta,t)| \, dt \\
		& \le C \norm{f}_{C^k(\T^d)} \int_0^\infty \frac{dt}{(1+t)^{1-\sigma} (1+\varkappa t)^{\sigma \ell}} \\
		& \le C \varkappa^{-\sigma} \norm{f}_{C^k(\T^d)}.
	\end{align*}
	This completes the proof.
\end{proof}

\begin{proof}[\bf Proof of Theorem \ref{thm_tauC}]

	{\bf Step 1: Set-up and reduction.} Let $n_1 = (n_{1,1},\cdots,n_{1,d})$, $n_2 =(n_{2,1},\cdots,n_{2,d})
	\in \mathbb{S}^{d-1}_D$
	and $\delta = |n_1 - n_2|>0$. Choose 
	 $d\times (d-1)$ matrices $N_1$, $N_2$ such that both $M_1 = (N_1,-n_1)$ and 
	 $M_2 = (N_2,-n_2)$ are orthogonal and $|N_1-N_2|\le C \delta$.
	  Let $U_1$, $U_2$ be  solutions of the systems in the form of (\ref{eq_Un}) associated with $n_1,n_2$, respectively, i.e.,
	\begin{equation}\label{eq_U1}
		\left\{
		\begin{aligned}
			- \Dr{N^T_1 \nabla_\theta}{\partial_t} \cdot B_1
			\Dr{N^T_1 \nabla_\theta}{\partial_t} U_1
			&= 0 &\quad & \txt{ in }\T^d\times (0,\infty), \\
			-e_{d+1}\cdot B_1 \Dr{N^T_1 \nabla_\theta}{\partial_t} U_1 &= T_{1,ij} \cdot \nabla_\theta \phi_{ij} &\quad & \txt{ on }  \T^d\times \{0\},
		\end{aligned}
		\right.
	\end{equation}
	and
	\begin{equation}\label{eq_U2}
		\left\{
		\begin{aligned}
			- \Dr{N^T_2 \nabla_\theta}{\partial_t} \cdot B_2
			\Dr{N^T_2 \nabla_\theta}{\partial_t} U_2
			&= 0  &\quad & \txt{ in } \T^d\times (0,\infty), \\
			-e_{d+1}\cdot B_2 \Dr{N^T_2 \nabla_\theta}{\partial_t} U_2 &= T_{2,ij} \cdot \nabla_\theta \phi_{ij} 
			&\quad & \txt{ on }  \T^d\times \{0\},
		\end{aligned}
		\right.
	\end{equation}
	where $T_{\ell,ij} = n_{\ell,i} e_j - n_{\ell,j} e_i$
	 are  vectors orthogonal to $n_\ell$ and $B_\ell(\theta,t) = M_\ell^T A^*(\theta-tn_\ell)M_\ell$
	  for $\ell = 1,2$.

Without loss of generality, we may assume that
$\varkappa = \varkappa(n_1)\ge \varkappa(n_2)$. 
In view of the formula (\ref{eq_tau}), we only need to show
that
\begin{equation}\label{key}
	\int_{\T^d}|T_{1,ij} \cdot \nabla_\theta U_1(\theta,0) - T_{2,ij} \cdot \nabla_\theta U_2(\theta,0)|^2 \, d\theta
	 \le C_\sigma \varkappa^{-2\sigma}|n_1 - n_2|^2 
\end{equation}
for $1\le i, j\le d$.
By the triangle inequality, 
\begin{align*}\label{key}
	& \int_{\T^d} |T_{1,ij} \cdot \nabla_\theta U_1(\theta,0) - T_{2,ij} \cdot \nabla_\theta U_2(\theta,0)|^2 \,
	d\theta \\
	&\qquad  \le 2\int_{\T^d} |(T_{1,ij} - T_{2,ij}) \cdot \nabla_\theta U_1(\theta,0) |^2 \, d\theta + 2\int_{\T^d}  |T_{2,ij} \cdot \nabla_\theta (U_1(\theta,0) - U_2(\theta,0))|^2 \, d\theta \\
	&\qquad \le C \varkappa^{-2\sigma} \delta^2 +  C\int_{\T^d} |N_2^T \nabla_\theta(U_1(\theta,0) - U_2(\theta,0))|^2 \, d\theta,
\end{align*}
where in the last inequality we have used (\ref{est_Ubd}) 
and the fact that the columns of $N_2$ span the subspace orthogonal to $n_2$. 
Furthermore, we let $W = U_1 - U_2$ and note that
\begin{equation}\label{est_N2Wt0}
\aligned
&	\int_{\T^d} |N_2^T \nabla_{\theta} W(\theta,0) |^2 \, d\theta\\
&\quad
 \le 2 \int_0^1 \int_{\T^d} |N_2^T \nabla_{\theta} W(\theta,t) |^2 \, d \theta dt +  2\int_0^1 \int_{\T^d} |N_2^T \nabla_{\theta} \partial_t W(\theta,t) |^2 \,d\theta dt.
 \endaligned
\end{equation}
As a result, it suffices to estimate the two terms in the RHS of the above inequality.

\medskip

{\bf Step 2: Estimate for $N^T_2 \nabla_\theta W$.} 

\medskip

The argument here is similar to that 
for Dirichlet problems, with Lemmas \ref{lem_Uexist}, \ref{lem_U} and \ref{lem_dtheU} in our disposal. 
Note that $W$ satisfies
\begin{equation}\label{eq_tWmain}
	\left\{
	\begin{aligned}
		- \Dr{N^T_2 \nabla_\theta}{\partial_t} \cdot B_2
		\Dr{N^T_2 \nabla_\theta}{\partial_t} W
		&= \Dr{N^T_2 \nabla_\theta}{\partial_t} \cdot G + H &\quad & \txt{ in } \T^d\times \R_+, \\
		-e_{d+1}\cdot B_2 \Dr{N^T_2 \nabla_\theta}{\partial_t} W &= e_{d+1} \cdot G +  (T_{1,ij} - T_{2,ij}) \cdot \nabla_\theta \phi_{ij} &\quad & \txt{ on }  \T^d\times \{0\},
	\end{aligned}
	\right.
\end{equation}
where $G = G_1 + G_2$ and $H$ are exactly the same as in (\ref{eq_G12H}) for Dirichlet problems. 

The following two lemmas were proved in  \cite{ShenZhuge16}.

\begin{lemma}\label{lem_A1est_N}
	Let $n\in \bS^{d-1}_D$ and $U$ be a solution of
	\begin{equation}
		\left\{
		\begin{aligned}
			-\Dr{N^T \nabla_\theta}{\partial_t}
			\cdot B \Dr{N^T \nabla_\theta}{\partial_t} U
			&= \Dr{N^T \nabla_\theta}{\partial_t} \cdot G  &\quad & \txt{ in }\T^d\times \R_+, \\
			-e_{d+1}\cdot B \Dr{N^T \nabla_\theta}{\partial_t} U &= e_{d+1}\cdot G  
			&\quad & \txt{ on }  \T^d\times \{0\}.
		\end{aligned}
		\right.
	\end{equation}
	Assume that
	\begin{equation}
		\sup_{t>0,\theta\in \T^d} (1+t)\Big\{ |N^T\nabla_\theta U(\theta,t)| + |\partial_t U(\theta,t)| + |G(\theta,t)| \Big\} < \infty.
	\end{equation}
	Then, for any $0<\sigma<1$,
	\begin{equation}
		\int_0^\infty \int_{\T^d} \left(|N^T \nabla_\theta U|^2 + |\partial_t U|^2\right) t^{\sigma-1} \, d\theta dt 
		\le C_\sigma \int_0^\infty \int_{\T^d} |G|^2 t^{\sigma-1} \, d\theta dt.
	\end{equation}
\end{lemma}

\begin{lemma}\label{lem_Loc_N}
	Let $n\in \bS^{d-1}_D$ and $U$ be a solution of
	\begin{equation}
		\left\{
		\begin{aligned}
			-\Dr{N^T \nabla_\theta}{\partial_t}
			\cdot B \Dr{N^T \nabla_\theta}{\partial_t} U
			&= 0  &\quad & \txt{ in } \T^d\times \R_+, \\
			-e_{d+1}\cdot B \Dr{N^T \nabla_\theta}{\partial_t} U &= h  &\quad & \txt{ on }  \T^d\times \{0\}.
		\end{aligned}
		\right.
	\end{equation}
	Assume that
	\begin{equation}\label{est_LocU}
		\sup_{t>0,\theta\in \T^d} (1+t)\Big\{ |N^T\nabla_\theta U(\theta,t)| + |\partial_t U(\theta,t)|\Big\} < \infty.
	\end{equation}
	Then,
	\begin{equation}
		\int_0^2 \int_{\T^d} \left(|N^T \nabla_\theta U|^2 + |\partial_t U|^2\right) 
		 d\theta dt \le C \int_{\T^d} |h|^2 \, d\theta.
	\end{equation}
\end{lemma}

Now we split $W$ as $W = W_1 + W_2 + W_3$, where
\begin{equation}\label{eq_W1_N}
	\left\{
	\begin{aligned}
		- \Dr{N^T_2 \nabla_\theta}{\partial_t} \cdot B_2
		\Dr{N^T_2 \nabla_\theta}{\partial_t} W_1
		&= 0  &\quad & \txt{ in } \T^d\times \R_+, \\
		-e_{d+1}\cdot B_2 \Dr{N^T_2 \nabla_\theta}{\partial_t} W_1 &= (T_{1,ij} - T_{2,ij}) \cdot \nabla_\theta \phi_{ij}  &\quad & \txt{ on } \T^d\times \{0\},
	\end{aligned}
	\right.
\end{equation}

\begin{equation}
	\left\{
	\begin{aligned}
		- \Dr{N^T_2 \nabla_\theta}{\partial_t} \cdot B_2
		\Dr{N^T_2 \nabla_\theta}{\partial_t} W_2
		&= \Dr{N^T_2 \nabla_\theta}{\partial_t} \cdot G  &\quad & \txt{ in } \T^d\times \R_+, \\
		-e_{d+1}\cdot B_2 \Dr{N^T_2 \nabla_\theta}{\partial_t} W_2 &= e_{d+1} \cdot G &
		\quad & \txt{ on }  \T^d\times \{0\},
	\end{aligned}
	\right.
\end{equation}

and

\begin{equation}
	\left\{
	\begin{aligned}
		- \Dr{N^T_2 \nabla_\theta}{\partial_t} \cdot B_2
		\Dr{N^T_2 \nabla_\theta}{\partial_t} W_3
		&= H  &\quad & \txt{ in } \T^d\times \R_+, \\
		-e_{d+1}\cdot B_2 \Dr{N^T_2 \nabla_\theta}{\partial_t} W_3 &= 0  &\quad & \txt{ on } 
		\T^d\times \{0\}.
	\end{aligned}
	\right.
\end{equation}

{\bf Estimate for $W_1$.} Since $\phi_{ij}$ is smooth, we can 
show that (\ref{eq_W1_N}) is solvable and the solution $W_1$ satisfies (\ref{est_LocU}). 
Thus, by Lemma \ref{lem_Loc_N}, 
\begin{equation}
\aligned
	\int_0^2 \int_{\T^d} \left(|N^T_2 \nabla_\theta W_1|^2 + |\partial_t W_1|^2\right)  d\theta dt  
	& \le C \int_{\T^d} |T_{1,ij} - T_{2,ij}|^2 |\nabla_\theta \phi_{ij}|^2 \, d\theta dt\\
	& \le C \delta^2.
	\endaligned
\end{equation}

{\bf Estimate for $W_2$.} By Lemma \ref{lem_A1est_N}, we have
\begin{align*}
	\int_0^\infty \int_{\T^d} \left(|N^T_2 \nabla_\theta W_2|^2 + |\partial_t W_2|^2\right) t^{\sigma-1} \, 
	d\theta dt & \le C \int_0^\infty \int_{\T^d} |G|^2 t^{\sigma-1} \, d\theta dt \\
	& \le C \sum_{k=1,2} \int_0^\infty \int_{\T^d} |G_k|^2 t^{\sigma-1} \, d\theta dt.
\end{align*}
Using (\ref{est_Ubd}) and (\ref{est_DUptwise}), 
we obtain
\begin{equation}\label{key}
	|\nabla_\theta U_1| \le C \varkappa^{-\sigma(1-\sigma)} [\varkappa^{-1}(1+\varkappa t)^{-\ell}]^{\sigma}
	 \le C \varkappa^{-2\sigma}(1+\varkappa t)^{-\sigma\ell}.
\end{equation}
Hence,
\begin{align*}\label{key}
	\int_0^\infty \int_{\T^d} |G_1|^2 t^{\sigma-1} \, d\theta  dt 
	& \le C \varkappa^{-4\sigma} \delta^2 \int_0^\infty (1+\varkappa t)^{-2\sigma \ell} t^{\sigma-1} \, dt \\
	& \le C \varkappa^{-5\sigma} \delta^2.
\end{align*}
Similarly, by (\ref{est_DrU}) and (\ref{est_DUptwise}), we have
\begin{equation}\label{key}
	|N^T_1 \nabla_\theta U_1| + |\partial_t U_1| \le C (1+t)^{1-\sigma} (1+\varkappa t)^{-\sigma\ell}.
\end{equation}
It follows that
\begin{align*}\label{key}
	\int_0^\infty \int_{\T^d} |G_2|^2 t^{\sigma-1} \, d\theta dt & \le
	C \delta^2 \int_0^\infty t^2 (1+t)^{2\sigma-2}(1+\varkappa t)^{-2\sigma \ell} t^{\sigma-1} \, dt \\
	& \le C \varkappa^{-3\sigma} \delta^2.
\end{align*}
As a result, we may conclude that
\begin{equation}\label{key}
	\int_0^\infty \int_{\T^d} \left(|N^T_2 \nabla_\theta W_2|^2 + |\partial_t W_2|^2\right) t^{\sigma-1} \, d\theta dt \le
	C \varkappa^{-5\sigma} \delta^2.
\end{equation}

{\bf Estimate for $W_3$.} The estimate for $W_3$ can be reduced to the first two cases. Let
\begin{equation}\label{key}
	\widetilde{H}(\theta,t) = -\int_t^\infty H(\theta,s) ds.
\end{equation}
Note that $\widetilde{H}$ is bounded for all $(\theta,t)\in \T^d\times \R_+$. Write
\begin{equation}\label{key}
	H(\theta,t) =\partial_t \widetilde{H}(\theta,t) = \Dr{N^T_2 \nabla_\theta}{\partial_t} \cdot \Dr{0}{\widetilde{H}(\theta,t)}.
\end{equation} 
Then, we can further decompose $W_3$ into $W_3 = W_{31} + W_{32}$, where
\begin{equation}
	\left\{
	\begin{aligned}
		- \Dr{N^T_2 \nabla_\theta}{\partial_t} \cdot B_2
		\Dr{N^T_2 \nabla_\theta}{\partial_t} W_{31}
		&= \Dr{N^T_2 \nabla_\theta}{\partial_t} \cdot \Dr{0}{\widetilde{H}(\theta,t)} &\quad & \txt{ in } \T^d\times 
		\R_+, \\
		-e_{d+1}\cdot B_2 \Dr{N^T_2 \nabla_\theta}{\partial_t} W_{32} &= e_{d+1}\cdot \Dr{0}{\widetilde{H}(\theta,t)}  &\quad & \txt{ on } \T^d\times \{0\},
	\end{aligned}
	\right.
\end{equation}
and
\begin{equation}
	\left\{
	\begin{aligned}
		- \Dr{N^T_2 \nabla_\theta}{\partial_t} \cdot B_2
		\Dr{N^T_2 \nabla_\theta}{\partial_t} W_{32}
		&= 0  &\quad & \txt{ in } \T^d\times \R_+, \\
		-e_{d+1}\cdot B_2 \Dr{N^T_2 \nabla_\theta}{\partial_t} W_{32} &= - e_{d+1}\cdot \Dr{0}{\widetilde{H}(\theta,t)} & \quad & \txt{ on }  \T^d\times \{0\}.
	\end{aligned}
	\right.
\end{equation}

Now by applying Lemma \ref{lem_A1est_N} for $W_{31}$, we obtain
\begin{equation}\label{key}
	\int_0^\infty \int_{\T^d} \left(|N^T_2 \nabla_\theta W_{31}|^2 + |\partial_t W_{31}|^2\right) 
	t^{\sigma-1}\, d\theta dt \le C\int_0^\infty \int_{\T^d} |\widetilde{H}|^2 t^{\sigma-1}\,  d\theta dt.
\end{equation}
It follows from Hardy's inequality (see \cite[p.272]{Stein}) that
\begin{align*}\label{key}
	\int_0^\infty \int_{\T^d} |\widetilde{H}|^2 t^{\sigma-1} \, d\theta dt & = \int_{\T^d} \int_0^\infty \bigg|\int_t^\infty H(\theta,s) ds \bigg|^2 t^{\sigma-1} \, dt d\theta \\
	& \le  \frac{4}{(1-\sigma)^2} \int_{\T^d} \int_0^\infty |H(\theta,t)|^2 t^{\sigma-1+2} \, dt d\theta.
\end{align*}
Consequently,
\begin{equation*}
	\int_0^\infty \int_{\T^d} \left(|N^T_2 \nabla_\theta W_{31}|^2 
	+ |\partial_t W_{31}|^2\right) t^{\sigma-1} \, d\theta dt
	\le C  \int_0^\infty \int_{\T^d} |H|^2 t^{1+\sigma}\, d\theta dt.
\end{equation*}
For $W_{32}$, using Lemma \ref{lem_Loc_N} and H\"{o}lder's inequality, we have
\begin{align*}
	\int_0^2 \int_{\T^d} & \left(|N^T_2 \nabla_\theta W_{32}|^2 + |\partial_t W_{32}|^2\right) d\theta dt \\
	& \le C \int_{\T^d} |\widetilde{H}(\theta,0)|^2\,  d\theta \\
	& \le C \int_{\T^d} \bigg| \int_0^\infty |H(\theta,t)| dt\bigg|^2 \, d\theta \\
	& \le C \int_{\T^d} \int_0^\infty |H(\theta,t)|^2 (1+t)^{2-\alpha} dt \int_0^\infty (1+t)^{\alpha-2} \, dt  d\theta \\
	& \le C  \int_0^\infty\int_{\T^d} (1+t)^2 |H(\theta,t)|^2 t^{-\alpha}\,  d\theta dt.
\end{align*}
Therefore, 
\begin{align*}
	\int_0^2 \int_{\T^d}  & \left(|N^T_2 \nabla_\theta W_3|^2 + |\partial_t W_3|^2\right) d\theta dt \\
	&\le C \int_0^\infty \int_{\T^d}  (1+t)^2 |H|^2 t^{\sigma-1} \, d\theta dt \\
	&\le C \delta^2 \int_0^\infty (1+t)^{2-2(\sigma-1)^2} (1+\varkappa t)^{-2\sigma \ell} t^{\sigma-1} \, dt \\
	& \le C  \varkappa^{-5\sigma} \delta^2,
\end{align*}
where in the last inequality we have chosen $\ell\ge 2$.

Summing up the estimates for $W_k$, we arrive at
\begin{equation}
	\int_0^2 \int_{\T^d} \left(|N^T_2 \nabla_\theta W|^2 + |\partial_t W|^2\right) d\theta dt \le C \varkappa^{-5\sigma} \delta^2,
\end{equation}
which proves the first part of (\ref{est_N2Wt0}), as $\sigma\in (0,1)$ can be arbitrarily small.

\medskip

{\bf Step 3: Estimate for $N^T_2 \nabla_\theta \partial_t W$.} 

\medskip

The argument is similar to Step 3 in
 the proof of Theorem \ref{thm_MU}. 
 Let $N_{2k}$ denote the $k$th column of $N_2$,
  and define the $k$th component of $N^T_2\nabla_\theta$ 
  by $\nabla_{2k} = N_{2k}^T\cdot \nabla_\theta$. We apply $\nabla_{2k}$ to (\ref{eq_tWmain}) and obtain
\begin{equation}\label{eq_d2jW}
	\left\{
	\begin{aligned}
		- \Dr{N^T_2 \nabla_\theta}{\partial_t} \cdot B_2
		\Dr{N^T_2 \nabla_\theta}{\partial_t} \nabla_{2k}W
		&= \Dr{N^T_2 \nabla_\theta}{\partial_t} \cdot \nabla_{2k}G + \nabla_{2k}H \\
		&\qquad +\Dr{N^T_2 \nabla_\theta}{\partial_t} \cdot \nabla_{2k}B_2
		\Dr{N^T_2 \nabla_\theta}{\partial_t} W  \\
		&\qquad \txt{ in } \T^d\times \R_+, \\
		-e_{d+1}\cdot B_2 \Dr{N^T_2 \nabla_\theta}{\partial_t} \nabla_{2k}W &= e_{d+1} \cdot \nabla_{2k}G +  \nabla_{2k}h \\
		&\qquad + e_{d+1}\cdot \nabla_{2k}B_2
		\Dr{N^T_2 \nabla_\theta}{\partial_t} W  \\
		&\qquad \txt{ on }  \T^d\times \{0\},
	\end{aligned}
	\right.
\end{equation}
where $h = (T_{1,ij} - T_{2,ij}) \cdot \nabla_\theta f_{ij}$. 
Let $\eta(t)$ be a cut-off function such that $\eta(t) = 1$ for $t\in [0,1]$,
 $\eta(t) = 0$ for $t\in [2,\infty)$,
 $0\le \eta(t)\le 1$ and $|\nabla \eta |\le C$. 
Now by integrating (\ref{eq_d2jW}) against $\nabla_{2k}(W \eta^2)$, we derive from integration by parts that
\begin{align*}
	&\int_0^1 \int_{\T^d} \left(|N^T_2 \nabla_\theta \nabla_{2k} W|^2 + |\partial_t \nabla_{2k} W|^2\right) d\theta dt \\
	&\qquad \le C \int_0^2 \int_{\T^d}
	\left(|\nabla_{2k}G|^2 + |\nabla_{2k}H|^2 +|N_2^T \nabla_\theta W|^2 +|\partial_t W|^2\right) d\theta dt + 
	C\norm{h}_{H^1(\T^d)}^2\\
	&\qquad \le C \varkappa^{-5\sigma} \delta^2.
\end{align*}
Consequently,
\begin{equation}\label{key}
	\int_0^1 \int_{\T^d} \left(|N^T_2 \nabla_\theta\otimes N^T_2\nabla_\theta W|^2
	 + |\partial_t N^T_2 \nabla_\theta W|^2\right) d\theta dt \le C \varkappa^{-5\sigma} \delta^2,
\end{equation}
which finishes the proof.
\end{proof}

\begin{proof}[\bf Proof of Theorem \ref{main-thm-02}]
With Theorem \ref{thm_tauC} at our disposal, the proof of Theorem \ref{main-thm-02} is identical to
 that of Theorem \ref{main-thm-01}.
\end{proof}

\bibliographystyle{amsplain}
\bibliography{Shen-Zhuge-4.bbl}

\begin{flushleft}
Zhongwei Shen,
 Department of Mathematics,
University of Kentucky,
Lexington, Kentucky 40506,
USA. \ \ 
% Fax: 1-859-257-4078.

E-mail: zshen2@uky.edu
\end{flushleft}

\begin{flushleft}
Jinping Zhuge, 
Department of Mathematics,
University of Kentucky,
Lexington, Kentucky 40506,
USA. \ \ 

E-mail: jinping.zhuge@uky.edu

\end{flushleft}
\medskip

\noindent \today

\end{document}